\documentclass[a4paper]{article}
\usepackage{amsmath, color, authblk, layout}
\usepackage[hyperref, usenames, svgnames,x11names, dvipsnames]{xcolor}
\usepackage{amsfonts, mathtools}
\usepackage{amssymb}
\usepackage{amsthm}
\usepackage[ocgcolorlinks, colorlinks=true, citecolor=Cerulean, linkcolor=Bittersweet, urlcolor=blue]{hyperref}

\numberwithin{equation}{section}

\newtheorem{theorem}{Theorem}[section]
\newtheorem{corollary}[theorem]{Corollary}
\newtheorem{lemma}[theorem]{Lemma}

\theoremstyle{definition}
\newtheorem{definition}[theorem]{Definition}

\newtheorem{remark}[theorem]{Remark}
 \newtheorem{ass}[theorem]{Assumption} 
  \newtheorem{defnass}[theorem]{Assumption and Definition} 
  \newtheorem{asss}[theorem]{Assumptions} 

\newcommand{\norm}[2]{\left\lVert #1\right\rVert_{#2}}
\newcommand{\md}{\partial^\bullet}
\newcommand{\Vs}{V_0}
\newcommand{\Vt}{V(t)}

\newcommand{\Vms}{V_0^{*}}
\newcommand{\Vmt}{V^{*}(t)}
\newcommand{\Hs}{H_0}
\newcommand{\Ht}{H(t)}
\newcommand{\Hms}{H_0^{*}}
\newcommand{\Hmt}{H^{*}(t)}
\newcommand{\Xs}{X_0}
\newcommand{\Xt}{X(t)}
\newcommand{\Xms}{X_0^{*}}
\newcommand{\Xmt}{X^{*}(t)}
\newcommand{\grad}{\nabla}

\newcommand{\sgradt}{\nabla_{\Gamma(t)}}

\newcommand{\slapt}{\Delta_{\Gamma(t)}}
\newcommand{\weaklyto}{\rightharpoonup}

\newcommand{\symbolForLittlec}{\lambda}
\newcommand{\symbolForBigC}{\Lambda}

\title{An abstract framework for parabolic PDEs on evolving spaces}
\author{Amal Alphonse, Charles M. Elliott, and Bj\"orn Stinner}
\affil{Mathematics Institute\\ University of Warwick\\ Coventry CV4 7AL\\ United Kingdom}
\begin{document}

\maketitle

\begin{abstract}
We present an abstract framework for treating the theory of well-posedness of solutions to abstract parabolic partial differential equations on evolving Hilbert spaces. This theory is applicable to variational formulations of PDEs on evolving spatial domains including moving hypersurfaces. We formulate an appropriate time derivative on evolving spaces called the material derivative and define a weak material derivative in analogy with the usual time derivative in fixed domain problems; our setting is abstract and not restricted to evolving domains or surfaces. Then we show well-posedness to a certain class of parabolic PDEs under some assumptions on the parabolic operator and the data.
\end{abstract}

\section{Introduction}
Partial differential equations on evolving or moving domains are an {active} area of research \cite{CorRod13}, \cite{actanumerica}, \cite{reusken}, \cite{similar}, partly because their study leads to interesting analysis but {also because models describing applications such as biological and physical phenomena can be better formulated on evolving domains (including hypersurfaces) rather than on stationary domains. For example, see \cite{Barreira2011}, \cite{Venkataraman2011} for studies of pattern formation on evolving surfaces, \cite{Garcke2014} for the modelling of surfactants in two-phase flows, \cite{Eilks2008} for the modelling and numerical simulation of dealloying by surface dissolution of a binary alloy (involving a forced mean curavture flow coupled to a Cahn--Hilliard equation), \cite{Elliott2009} (and the references therein for applications) for the analysis of a diffuse interface model for a linear surface PDE, and \cite{ElliottStinnerVankataraman} for the modelling and simulation of cell motility.}

One aspect to consider in the study of such equations is how to formulate the space of functions that have domains which evolve in time. Taking a disjoint union of the domains in time to form a non-cylindrical set is standard: see \cite{Bonaccorsi}, \cite{Euclid}, \cite{similar} for example. Of particular interest is \cite{Kloeden} where the problem of a semilinear heat equation on a time-varying domain is considered; the set-up of the evolution of the domains is comparable to ours and similar function space results are shown (in the setting of Sobolev spaces). In \cite{meier}, the authors define Bochner-type spaces by considering a continuous distribution of domains $\{\Gamma(t)\}_{t \in [0,T]} \subset \mathbb{R}^n$ that are embedded in a larger domain $\Gamma$. The aim of our work is to accommodate not only evolving domains but arbitrary evolving spaces. Our method, which follows that of \cite{vierling}, is somewhat different to the aforementioned and contains an attachment to standard Bochner spaces in a fundamental way. 

A common procedure for showing well-posedness of equations on evolving domains involves a transformation of the PDE onto a fixed reference domain to which abstract techniques from functional analysis are applied \cite{LipIFB}, \cite{1187}, \cite{ALE}, \cite{vierling}. For example, in \cite{vierling}, the heat equation 
\begin{equation}\label{eq:introductionHeatEquation}
\dot u(t) - \slapt u (t) + u(t)\sgradt \cdot \mathbf w(t)= f(t)\qquad \text{in $H^{-1}(\Gamma(t))$}
\end{equation}
on an evolving surface $\{\Gamma(t)\}_{t \in [0,T]}$ is considered, with $\mathbf w$ representing the velocity field. The equation is pulled back onto a reference domain $\Gamma(s)$ and standard results on linear parabolic PDEs are applied. A Faedo--Galerkin method (see \cite{benzitoscano} for a historical overview of the method) is used in \cite{1187} (for a different PDE), where the evolving domain is represented by the evolution of a perturbation of the reference domain and \emph{a~priori} estimates are derived for a linearised problem. An adapted Galerkin method that uses the pushforward of eigenfunctions of the Laplace--Beltrami operator on $\Gamma(0)$ to form a countable dense subset of $H^1(\Gamma(t))$ is employed  in \cite{dziuk_elliott} for the advection-diffusion equation \eqref{eq:introductionHeatEquation}. We abstract this  approach for one of our results. Well-posedness for the same class of equations is obtained in \cite{reusken} by employing a variational formulation on space-time surfaces  and utilising a standard generalisation of the classical Lax--Milgram theorem used by Lions for parabolic equations. We also employ this Lions--Lax--Milgram approach in our abstract setting.

As we have seen, there is much literature in which certain equations on evolving domains are studied, however, to the best of our knowledge, there is no unifying theory or framework that treats parabolic PDEs on \emph{abstract} evolving spaces.
The main aim of this paper is to provide this abstract framework. More specifically, given a linear time-dependent operator $A(t)$ we study well-posedness of parabolic problems of the form
\begin{align}
\dot u(t) + A(t)u(t) &= f(t)\qquad\text{}\label{eq:operatorEquationIntro}
\end{align}
as an equality in $V^*(t)$, with $V(t) \subset H(t)$ a Hilbert space for each $t \in [0,T]$. A main feature of our work is the definition of an appropriate time derivative on evolving spaces \emph{in an abstract setting}. When the said spaces are simply $L^p$ spaces on curved or flat surfaces in $\mathbb{R}^n$ that evolve in time, it is commonplace to take the material derivative
\[\dot u(t) = u_t(t) + \grad u(t)\cdot \mathbf w(t)\]
from continuum mechanics as the natural time derivative. But when we have arbitrary spaces that may have no relationship whatsoever with $\mathbb{R}^n$ it is not at all clear what the $\dot u(t)$ in \eqref{eq:operatorEquationIntro} should mean. We will deal with this issue and define a material derivative and a weak material derivative for the abstract case. Our framework relies on the existence of a family of (pushforward) maps $\phi_t$ for $t \in [0,T]$ that allow us to map the initial spaces $V(0)$ and $H(0)$ to the spaces $V(t)$ and $H(t)$. A particular realisation of these maps $\phi_t$ in the case of, for example,  the heat equation (\ref{eq:introductionHeatEquation}) takes into account the evolution of the surfaces $\Gamma(t)$ and hence $\phi_t$ will be  a flow map defined by  the velocity field $\mathbf w.$ {Although one motivation behind this work is the analysis of equations on moving domains and hypersurfaces, the framework can also be useful for problems on fixed domains where, for example, $H(t)$ and $V(t)$ may be weighted Lebesgue--Sobolev spaces with time-dependent weights.}

{Our belief is that the abstract procedure presented in this work is a clean and elegant approach to problems on moving domains. In addition, the theory and concepts presented here can be used as a foundation in extensions such as generalisations to the Banach space setting and the study of nonlinear problems.}
We also anticipate that our framework will benefit those working in numerical analysis since curved, flat, and evolving surfaces can all be treated with the same abstract procedure. 

{In a forthcoming paper \cite{AlpEllStiApplications}, we will demonstrate the applicability of this abstract framework to the case of moving or evolving hypersurfaces. Four different examples of parabolic equations on moving hypersurfaces will be studied, and the well-posedness will be proved using the results we shall give here.}
\subsection{Outline}
In \S \ref{sec:function_spaces}, we start by setting up the function spaces and definitions required for the analysis and indeed the \emph{statement} of equations of the form \eqref{eq:operatorEquationIntro}. We state our assumptions on the evolution of the spaces and define abstract strong and weak material derivatives (in analogy with the usual derivative and weak derivative utilised in fixed domain problems).

In \S \ref{sec:main}
we  precisely formulate the problem \eqref{eq:operatorEquationIntro} that we consider and list the assumptions we make on $A$. Statements of the main theorems of existence, uniqueness, and regularity of solutions are given. The proof of one of these theorems is presented in \S \ref{sec:proofs}. There, we make use of the generalised Lax--Milgram theorem. In \S \ref{sec:galerkinApproximations} we formulate an adapted abstract Galerkin method similar to one described in \cite{dziuk_elliott} and use it to prove a regularity result.

\subsection{Notation and conventions}
Here and below we fix $T \in (0,\infty)$. When we write expressions such as $\phi_{(\cdot)}u(\cdot)$, our intention usually (but not always) is that both of the dots $(\cdot)$ denote the same argument; for example,
$\phi_{(\cdot)}u(\cdot)$ will come to mean the map $t \mapsto \phi_tu(t).$ The notation $X^*$ will denote the dual space of a Hilbert space $X$ and $X^*$ will be equipped with the usual induced norm $\norm{f}{X^*} = \sup_{x \in X\backslash \{0\}} \langle f, x \rangle_{X^*,X} / \norm{x}{X}$. We may reuse the same constants in calculations multiple times if their exact value is not relevant. Integrals will usually be written as $\int_S f(s)$ instead of $\int_S f(s)\;\mathrm{d}s$ unless to avoid ambiguity. Finally, we shall make use of standard notation for Bochner spaces; for example, see \cite[\S 5.9]{evans}.
\section{Function spaces}\label{sec:function_spaces}
As we mentioned above, in order to properly understand and express the equation \eqref{eq:operatorEquationIntro}, we need to devise appropriate spaces of functions. First, we begin with recalling some standard results regarding Sobolev--Bochner spaces from parabolic theory for the reader's convenience; a good reference for this is \cite[\S XVIII]{DautrayLions}.

\subsection{Standard Sobolev--Bochner space theory}\label{sec:1.1}
Let $\mathcal V$ and $\mathcal H$ be Hilbert spaces and let $\mathcal{V} \subset \mathcal H \subset \mathcal V^*$ be a Gelfand triple (i.e., all embeddings are continuous and dense and $\mathcal H$ is identified with its dual via the Riesz representation theorem). Recall that $u \in L^2(0,T;\mathcal V)$ is said to have a \emph{weak derivative} $u' \in L^2(0,T;\mathcal V^*)$ if there exists $w\in L^2(0,T;\mathcal V^*)$ such that

\begin{equation}\label{eq:defnWeakDer}
\int_0^T \zeta'(t)(u(t),v)_{\mathcal H} = -\int_0^T \zeta(t)\langle w(t),v\rangle _{\mathcal V^{*},\mathcal V} \qquad \text{for all $\zeta \in \mathcal{D}(0,T)$ and $v\in \mathcal V$},
\end{equation}
and one writes $w=u'$. By $\mathcal{D}(0,T)$ we refer to the space of infinitely differentiable functions with compact support in $(0,T)$. We shall also make use of $\mathcal{D}([0,T];\mathcal V)$; this is the space of {$\mathcal V$-valued} infinitely differentiable functions compactly supported in the \emph{closed} interval $[0,T]$. A helpful characterisation of this space, from Lemma 25.1 in \cite[\S IV.25]{wloka}, is that $\mathcal{D}([0,T];\mathcal V)$ is the restriction $\mathcal{D}((-\infty, \infty);\mathcal V)|_{[0,T]}$ (the restriction to $[0,T]$ of infinitely differentiable {$\mathcal{V}$-valued} functions with compact support).
\begin{lemma}\label{lem:spaceW0}
The space 
\[
\mathcal W(\mathcal V, \mathcal V^*) = \{u \in L^2(0,T;\mathcal V) \mid u' \in L^2(0,T; \mathcal V^*)\}
\]
with inner product 
\[ (u,v)_{\mathcal W(\mathcal V, \mathcal V^*)} = \int_0^T (u(t), v(t))_{\mathcal V} + \int_0^T(u'(t), v'(t))_{\mathcal V^*} \]
is a Hilbert space. Furthermore,
\begin{itemize}
\item[1.] The embedding $\mathcal W(\mathcal V, \mathcal V^*) \subset C([0,T]; \mathcal H)$ is continuous. 
\item[2.]The embedding $\mathcal{D}([0,T];\mathcal V) \subset \mathcal W(\mathcal V, \mathcal V^*)$ is dense.
\item[3.]For $u$, $v \in \mathcal W(\mathcal V, \mathcal V^*)$, the map $t \mapsto (u(t), v(t))_{\mathcal H}$ is absolutely continuous on $[0,T]$ and
\[
\frac{d}{dt}(u(t), v(t))_{\mathcal H} = \langle u'(t), v(t) \rangle_{\mathcal V^*, \mathcal V} + \langle u(t), v'(t) \rangle_{\mathcal V, \mathcal V^*}
\]
for almost every $t \in [0,T]$, hence the integration by parts formula
\begin{equation*}\label{eq:FOPI}
(u(T), v(T))_{\mathcal H} - ( u(0), v(0) )_{\mathcal H}=     \int_0^T \langle u'(t), v(t) \rangle_{\mathcal V^*, \mathcal V} + \int_0^T\langle u(t), v'(t) \rangle_{\mathcal V, \mathcal V^*}
\end{equation*}
holds.
\end{itemize}
\end{lemma}
\begin{proof}
The density result is Theorem 2.1 in \cite[\S 1.2]{lionsmagenes}. For the rest, consult Proposition 1.2 and Corollary 1.1 in \cite[\S III.1]{showalter}.
\end{proof}
We can characterise the weak derivative in terms of vector-valued test functions. This is useful because it more closely resembles the weak material derivative that we shall define later on.
\begin{theorem}[Alternative characterisation of the weak derivative]
The weak derivative condition \eqref{eq:defnWeakDer}
is equivalent to
\begin{equation*}\label{eq:defnWeakDerAlt}
\int_0^T (u(t), \psi'(t) )_{\mathcal H} = -\int_0^T \langle u'(t), \psi(t) \rangle_{\mathcal V^*, \mathcal V} \qquad \text{for all $\psi \in \mathcal{D}((0,T); \mathcal V)$}.
\end{equation*}
\end{theorem}

We finish this subsection with some words on measurability.
\begin{definition}[Weak measurability]\label{defn:weakMeasurability}Let $X$ be a Hilbert space. A function $f\colon [0,T] \to X$ is \emph{weakly measurable} if for every $x \in X$, the map 
$t \mapsto (f(t), x)_X$ 
is measurable on $[0,T]$.
\end{definition}
Strong (or Bochner) measurability implies weak measurability. If the Hilbert space $X$ turns out to be separable, then both notions of measurability are equivalent thanks to Pettis' theorem \cite[\S 1.5, Theorem 1.34]{roubicek}.

\subsection{Evolving spaces}\label{sec:assumptionsEvolution}
Now we shall define Bochner-type function spaces to treat evolving spaces. We start with some notation and concepts  on the evolution itself. We informally identify a family of Hilbert spaces $\{X(t)\}_{t \in [0,T]}$ with the symbol $X$, and given a family of maps $\phi_{t}\colon \Xs \to \Xt$ 
we define the following notion of \textbf{compatibility} of the pair $(X, (\phi_{t})_{t \in [0,T]})$. 
\begin{definition}[Compatibility]\label{comp}
We say that a pair $(X, (\phi_{t})_{t \in [0,T]})$ is \emph{compatible} if all of the following conditions hold.

For each $t \in [0,T]$, $X(t)$ is a real separable Hilbert space (with $X_0 := X(0)$) and the map
\[\phi_{t}\colon X_0 \to X(t)\]
is a linear homeomorphism such that $\phi_0$ is the identity. By $\phi_{-t}\colon \Xt \to \Xs$ we denote the inverse of $\phi_t.$   Furthermore, we will assume that there exists a constant $C_X$ independent of $t \in [0,T]$ such that
\begin{equation*}
\begin{aligned}
\norm{\phi_t u}{X(t)} &\leq C_X\norm{u}{X_0}&&\forall u \in X_0\\
\norm{\phi_{-t} u}{\Xs} &\leq C_X\norm{u}{\Xt}&&\forall u \in \Xt.
\end{aligned}
\end{equation*}
Finally, we assume that the map
\begin{equation*}
\begin{aligned}
&t \mapsto \norm{\phi_t u}{X(t)}\qquad &&\forall u \in X_0
\end{aligned}
\end{equation*}
is continuous. 
\end{definition}
We often write the pair as $(X, \phi_{(\cdot)})$ for convenience. We call $\phi_{t}$ and $\phi_{-t}$ the \emph{pushforward} and \emph{pullback} maps respectively. In the following we will assume compatibility of $(X, \phi_{(\cdot)})$. As a consequence of these assumptions, we have that the dual operator of $\phi_t$, denoted $\phi_t^*\colon X^*(t) \to X_0^*$, is itself a linear homeomorphism, as is its inverse $\phi_{-t}^*\colon \Xms \to \Xmt$, and they satisfy
\begin{equation*}
\begin{aligned}
\norm{\phi_t^* f}{X_0^*} &\leq C_X\norm{f}{X^*(t)} &&\forall f \in X^*(t)\\
\norm{\phi_{-t}^* f}{\Xmt} &\leq C_X\norm{f}{\Xms} &&\forall f \in \Xms.
\end{aligned}
\end{equation*}
By separability of $X_0$, it also follows that the map
\begin{align*}
t \mapsto \norm{\phi_{-t}^* f}{X^*(t)}\quad \forall f \in X^*_0
\end{align*}
is measurable.
\begin{remark}
{If we define $U(t,s)\colon X(s) \to X(t)$ by $U(t,s) := \phi_t \phi_{-s}$ for $s$, $t \in [0,T]$, it can be readily seen from 
$U(t,r)U(r,s) = \phi_t \phi_{-r}\phi_r \phi_{-s} = \phi_t\phi_{-s} = U(t,s)$ that the family of operators $U(t,s)$ is a two-parameter semigroup. }
\end{remark}
\begin{remark}
Note that the above implies the equivalence of norms
\begin{equation*}
\begin{aligned}
C_X^{-1}\norm{u}{\Xs} &\leq \norm{\phi_t u}{\Xt} \leq C_X\norm{u}{\Xs}&&\forall u \in \Xs,\\
C_X^{-1}\norm{f}{\Xmt} &\leq \norm{\phi_t^* f}{\Xms} \leq C_X \norm{f}{\Xmt} &&\forall f \in \Xmt.
\end{aligned}
\end{equation*}
\end{remark}
We now define appropriate time-dependent function spaces to handle functions defined on evolving spaces. Our spaces are generalisations of those defined in \cite{vierling}.
\begin{definition}[The spaces $L^2_X$ and $L^2_{X^*}$]
Define the spaces 
\begin{align*}
L^2_X &= \{u:[0,T] \to \!\!\!\!\bigcup_{t \in [0,T]}\!\!\!\! X(t) \times \{t\}, t \mapsto (\bar u(t), t) \mid \phi_{-(\cdot)} \bar u(\cdot) \in L^2(0,T;X_0)\}\\
L^2_{X^*} &= \{f:[0,T] \to \!\!\!\!\bigcup_{t \in [0,T]}\!\!\!\! X^*(t) \times \{t\},t \mapsto (\bar f(t), t) \mid \phi_{(\cdot)}^* \bar f(\cdot) \in L^2(0,T;X^*_0) \}.
\end{align*}
More precisely, these spaces consist of equivalence classes of functions agreeing almost everywhere in $[0,T]$, just like ordinary Bochner spaces. 
\end{definition}
We first show that these spaces are inner product spaces, and later we prove that they are in fact Hilbert spaces. For $u \in L^2_X$, we will make an abuse of notation and identify $u(t) = (\bar u(t), t)$ with $\bar u(t)$ (and likewise for $f \in L^2_{X^*}$).
\begin{theorem}\label{thm:ips}The spaces $L^2_X$ and $L^2_{X^*}$ are inner product spaces with the inner products
\begin{equation}\label{eq:innerProductOnL2X}
\begin{aligned}
(u, v)_{L^2_X} &= \int_0^T (u(t), v(t))_{X(t)}\;\mathrm{d}t\\
(f, g)_{L^2_{X^*}} &= \int_0^T (f(t), g(t))_{X^*(t)}\;\mathrm{d}t.\\
\end{aligned}
\end{equation}
\end{theorem}
\begin{proof}
It is easy to verify that the expressions in \eqref{eq:innerProductOnL2X} define inner products if the integrals on the right hand sides are well-defined, which we now check. For the $L^2_X$ case, it suffices to show that $\norm{u(t)}{X(t)}^2$ is integrable for every $u \in L^2_X$. So let $u \in L^2_X$. Then $\tilde u := \phi_{-(\cdot)}u(\cdot) \in L^2(0,T;X_0)$. Define $F\colon [0,T] \times X_0 \to \mathbb{R}$ by $F(t,x) = \norm{\phi_t x}{X(t)}$. By assumption, $t \mapsto F(t,x)$ is measurable for all $x \in X_0$, and if $x_n \to x$ in $X_0$, then by the reverse triangle inequality,
\begin{align*}
|F(t,x_n) - F(t,x)| \leq \norm{\phi_t(x_n-x)}{X(t)} \leq C_X\norm{x_n -x}{X_0} \to 0,
\end{align*} 
so $x \mapsto F(t,x)$ is continuous. Thus $F$ is a Carath\'eodory function. Due to the condition $|F(t,x)| \leq C_X\norm{x}{X_0}$, by Remark 3.4.5 of \cite{gasinski}, the Nemytskii operator $N_F$ defined by $(N_F x)(t) := F(t,x(t))$ maps  $L^2(0,T;X_0) \to L^2(0,T)$, so that
\[\norm{N_F \tilde u}{L^2(0,T)}^2 = \int_0^T \norm{u(t)}{X(t)}^2 < \infty.\]
This proves the theorem for $L^2_X$. The process is the same for the case of $L^2_{X^*}$ except we replace $\phi_{-t}$ and $\phi_t$ with the dual maps $\phi_{t}^*$ and $\phi_{-t}^*$.
\end{proof}
\begin{lemma}\label{cor:convergenceSimpleMeasurableFunctions}
Let $u \in L^2_X$ and $f \in L^2_{X^*}.$ Then there exist simple measurable functions $u_n \in L^2(0,T;X_0)$ and $f_n \in L^2(0,T;X_0^*)$ such that for almost every $t \in [0,T]$,
\begin{equation*}
\begin{aligned}[2]
\phi_tu_n(t)& \to u(t) \qquad &&\text{in $X(t)$}\\
\phi_{-t}^*f_n(t)& \to f(t) \qquad &&\text{in $X^*(t)$}
\end{aligned}
\end{equation*}
as $n \to \infty.$
\end{lemma}
This lemma can be proved by using the density of simple measurable functions in $L^2(0,T;X_0)$. The following result is required to show that the above spaces are complete.
\begin{lemma}[Isomorphism with standard Bochner spaces]\label{lem:pullbackIsInL2X}
The maps
\begin{alignat*}{2}
u &\mapsto \phi_{(\cdot)}u(\cdot) \qquad&&\text{from $L^2(0,T;X_0)$ to $L^2_X$}\\
f &\mapsto \phi_{-(\cdot)}^*f(\cdot)&&\text{from $L^2(0,T;X_0^*)$ to $L^2_{X^*}$}
\end{alignat*}
are both isomorphisms between the respective spaces.
\end{lemma}
%
%
For the proof of the $L^2_X$ case, one makes an argument similar to that in the proof of Theorem \ref{thm:ips} and shows that given an arbitrary $u \in L^2(0,T;X_0)$, the map $t \mapsto \norm{\phi_tu(t)}{X(t)}^2$ is indeed measurable (then it follows that $\norm{\phi_{(\cdot)}u(\cdot)}{L^2_X}$ is finite). That the spaces are isomorphic follows from the above (which shows that there is a map from $L^2(0,T;X_0)$ to $L^2_X$) and the definition of $L^2_X$. The isomorphism is $T\colon L^2(0,T;X_0) \to L^2_X$ where 
\[Tu = \phi_{(\cdot)}u(\cdot)\qquad\text{and}\qquad T^{-1}v = \phi_{-(\cdot)}v(\cdot).\]
It is easy to check that $T$ is linear and bijective. The proof for the $L^2_{X^*}$ case uses the same readjustments as before.

The next lemma, which is a consequence of the uniform bounds on $\phi_t$ and $\phi^*_t$, will be in constant use throughout this work.
\begin{lemma}\label{lem:equivalenceOfNorms}The equivalence of norms
\begin{equation*}
\begin{aligned}[2]
\frac{1}{C_X}\norm{u}{L^2_X} &\leq \norm{\phi_{-(\cdot)}u(\cdot)}{L^2(0,T;X_0)} \leq C_X\norm{u}{L^2_X}\qquad &&\forall u \in L^2_X\\
\frac{1}{C_X}\norm{f}{L^2_{X^*}} &\leq \norm{\phi_{(\cdot)}^*f(\cdot)}{L^2(0,T;X^*_0)} \leq C_X\norm{f}{L^2_{X^*}}&&\forall f \in L^2_{X^*}
\end{aligned}
\end{equation*}
holds.
\end{lemma}
\begin{corollary}The spaces $L^2_X$ and $L^2_{X^*}$ are separable Hilbert spaces.
\end{corollary}
\begin{proof}
Since $L^2_{X}$ and $L^2(0,T;X_0)$ are isomorphic and the latter space is complete, so too is $L^2_{X}$ by the equivalence of norms result in the previous lemma. The separability also follows from the previous lemma.
\end{proof}
We now investigate the relationship between the dual space of $L^2_X$ and the space $L^2_{X^*}.$ We in fact prove that these spaces can be identified; this requires the following preliminary lemmas.
\begin{lemma}\label{lem:dualPairingIsIntegrable}
For $f \in L^2_{X^*}$ and $u \in L^2_X$, the map
\[t \mapsto \langle f(t),u(t) \rangle_{\Xmt, \Xt}\]
is integrable on $[0,T].$
\end{lemma}
\begin{proof}
By considering the Carath\'eodory map $F\colon [0,T] \times X_0^* \times X_0 \to \mathbb{R}$ defined by $F(t,x^*,x) = \langle \phi_{-t}^*x^*, \phi_t x \rangle_{X^*(t), X(t)}$ and using Remark 3.4.2 of \cite{gasinski}, given $f \in L^2_{X^*}$ and $u \in L^2_X$, we have with $\tilde f := \phi_{(\cdot)}^*f(\cdot)$ and $\tilde u:= \phi_{-(\cdot)}u(\cdot)$ that $t \mapsto \langle \phi_{-t}^*\tilde f(t), \phi_t \tilde u(t) \rangle_{X^*(t), X(t)} = \langle f(t), u(t) \rangle_{X^*(t), X(t)}$ is measurable, since $t \mapsto \tilde f(t)$ and $t \mapsto \tilde u(t)$ are measurable. That the integral is finite is trivial.
\end{proof}
\begin{lemma}\label{lem:functionMeasurable}
Suppose that $f(t) \in \Xmt$ for almost every $t \in [0,T]$ with
\[\int_0^T \norm{f(t)}{\Xmt}^2 < \infty,\]
and that for every $u \in L^2_X$, the map $t \mapsto \langle f(t), u(t) \rangle_{\Xmt,\Xt}$ is measurable. Then $f \in L^2_{X^*}$.
\end{lemma}
\begin{proof}
{We have
$\langle f(t), u(t) \rangle_{\Xmt,\Xt} =\langle \phi_t^*f(t), \phi_{-t}u(t) \rangle_{\Xms,\Xs},$
and the left hand side is measurable,} hence the map
\begin{equation*}
t \mapsto \langle \phi_t^*f(t), \phi_{-t}u(t) \rangle_{\Xms,\Xs}
\end{equation*}
is measurable on $[0,T]$ for every $u \in L^2_X$.

Given $w \in X_0$, the element $u(\cdot) := \phi_{(\cdot)} w \in L^2_X,$ so we have
 (from Definition \ref{defn:weakMeasurability} or Footnote 80 in \cite[\S 1.4, p.~36]{roubicek2} for example) that $\phi_{(\cdot)}^*f(\cdot)\colon [0,T]\to\Xms$ is weakly measurable. Now, as remarked after Definition \ref{defn:weakMeasurability}, we use Pettis' theorem to conclude that $\phi_{(\cdot)}^*f(\cdot)$ is indeed strongly measurable. Hence we can compute
\[\norm{\phi_{(\cdot)}^*f(\cdot)}{L^2(0,T;\Xms)}^2 = \int_0^T \norm{\phi_t^*f(t)}{\Xms}^2 \leq C_X^2\int_0^T \norm{f(t)}{\Xmt}^2 < \infty,\]
so $\phi_{(\cdot)}^*f(\cdot) \in L^2(0,T;\Xms)$, giving $f \in L^2_{X^*}.$
\end{proof}
\begin{lemma}[Identification of $(L^2_X)^*$ and $L^2_{X^*}$]The spaces $(L^2_{X})^*$ and $L^2_{X^*}$ are isometrically isomorphic. Hence, we may identify $(L^2_X)^* \equiv L^2_{X^*}$, and the duality pairing of $f \in L^2_{X^*}$ with $u \in L^2_X$ is
\[\langle f, u \rangle_{L^2_{X^*}, L^2_X} = \int_0^T \langle f(t), u(t) \rangle_{\Xmt, \Xt}\;\mathrm{d}t.\]
\end{lemma}
\begin{proof}
Define the linear map $\mathcal J\colon L^2_{X^*} \to (L^2_{X})^*$ by 
\[\langle \mathcal{J}f, \cdot\rangle_{(L^2_X)^*, L^2_X} = \int_0^T \langle f(t), (\cdot)(t) \rangle_{\Xmt, \Xt}\;\mathrm{d}t.\]
This is well-defined due to Lemma \ref{lem:dualPairingIsIntegrable}. We must check that $\mathcal J$ is an isometric isomorphism.

Suppose that $F \in (L^2_X)^*$. We first need to show that there exists a unique $f \in L^2_{X^*}$ such that
$\mathcal{J}f = F.$
To do this, we use the Riesz map $\mathcal R\colon (L^2_X)^* \to L^2_X$ to write
\begin{equation}\label{eq:proofOfIsometricIsomorphism}
\langle F,u\rangle_{(L^2_X)^*, L^2_X} = (\mathcal RF,u)_{L^2_X} = \int_0^T (\mathcal RF(t),u(t))_{X(t)},
\end{equation}
and then with $\mathcal S_t^{-1}\colon \Xt \to \Xmt$ denoting the inverse Riesz map on $\Xt$, we get
\[(\mathcal RF(t),u(t))_{\Xt} = \langle \mathcal S_t^{-1}(\mathcal RF(t)), u(t) \rangle_{\Xmt,\Xt}\]
for almost all $t \in [0,T]$. Now, from \eqref{eq:proofOfIsometricIsomorphism}, the right hand side of this equality must be integrable. Hence
\[t \mapsto \langle \mathcal S_t^{-1}(\mathcal RF(t)), u(t) \rangle_{\Xmt,\Xt}\]
is measurable for every $u \in L^2_X.$ Now, the question is whether $\mathcal S_{(\cdot)}^{-1}(\mathcal RF(\cdot)) \in L^2_{X^*}.$ Clearly $\mathcal S_t^{-1}(\mathcal RF(t)) \in \Xmt$, and by the isometry of the Riesz maps,
\begin{equation}\label{eq:dualSpaceIsometry1}
\int_0^T \norm{\mathcal S_t^{-1}(\mathcal RF(t))}{\Xmt}^2 = \int_0^T\norm{\mathcal RF(t)}{\Xt}^2 = \norm{\mathcal RF}{L^2_X}^2 = \norm{F}{(L^2_X)^*}^2
\end{equation}
which is finite. Therefore, we obtain $\mathcal S_{(\cdot)}^{-1}(\mathcal RF(\cdot)) \in L^2_{X^*}$ by Lemma \ref{lem:functionMeasurable}. So $\mathcal{J}(\mathcal S_{(\cdot)}^{-1}\mathcal RF(\cdot)) = F$.

For uniqueness, suppose that $\mathcal{J}f = 0$. Then
\begin{align*}
\langle \mathcal{J}f, u\rangle_{(L^2_X)^*, L^2_X} &= \int_0^T \langle f(t), u(t) \rangle_{\Xmt, \Xt}\\
&= \int_0^T \langle \phi_t^*f(t), \phi_{-t}u(t) \rangle_{\Xms, \Xs}\\
&= \langle \phi_{(\cdot)}^*f(\cdot), \hat u \rangle_{L^2(0,T;\Xms), L^2(0,T;\Xs)}\tag{with $\hat u = \phi_{-(\cdot)}u(\cdot)$},
\end{align*}
which holds for all $\hat u \in L^2(0,T;\Xs)$. This implies that $f = 0$.

To see that $\mathcal{J}$ is an isometry, we define $\mathcal{J}^{-1}\colon (L^2_X)^* \to L^2_{X^*}$ by $\mathcal{J}^{-1}F = \mathcal S_{(\cdot)}^{-1}\mathcal RF(\cdot)$ and use \eqref{eq:dualSpaceIsometry1} to conclude.
\end{proof}
Although we have no notion of continuity in time for a function $u \in L^2_X$, we can nevertheless make the following definition.
\begin{definition}[Spaces of pushed-forward  continuously differentiable functions]
Define
\begin{align*}
C^k_X &= \{\xi \in L^2_X \mid \phi_{-(\cdot)}\xi(\cdot) \in C^k([0,T];X_0)\}\quad\text{for $k \in \{0,1,...\}$}\\
\mathcal{D}_X(0,T) &= \{\eta \in L^2_X \mid \phi_{-(\cdot)}\eta(\cdot) \in \mathcal{D}((0,T);X_0)\}\\
\mathcal{D}_X[0,T] &= \{\eta \in L^2_X \mid \phi_{-(\cdot)}\eta(\cdot) \in \mathcal{D}([0,T];X_0)\}.
\end{align*}
\end{definition}
Since $\mathcal{D}((0,T);X_0) \subset \mathcal{D}([0,T];X_0)$, we have 
\[\mathcal{D}_X(0,T) \subset \mathcal{D}_X[0,T] \subset C^k_X.\]

\subsection{Evolving Hilbert space structure}
In the preceding, we set up a Hilbert space $L^2_X$ and its dual $L^2_{X^*}$ based on an arbitrary family of separable Hilbert spaces $\{X(t)\}_{t \in [0,T]}$ and a suitable family of maps $\{\phi_t\}_{t \in [0,T]}$. 
We now lay the groundwork for posing PDEs on evolving spaces. For each $t \in [0,T]$, let $\Vt$ and $\Ht$ be (real) separable Hilbert spaces with $V_0 := V(0)$ and $H_0 := H(0)$ such that $V_0 \subset H_0$ is a continuous and dense embedding. Identifying $H_0$ with its dual space $H_0^*$, it follows that $H_0 \subset V_0^*$ is also continuous and dense. In other words, $V_0 \subset H_0 \subset V_0^*$ is a Gelfand or evolution triple of Hilbert spaces (i.e., a Hilbert triple) \cite[\S 7.2]{roubicek}.
\begin{asss}\label{asss:compatibilityOfEvolvingHilbertTriple}We will assume compatibility in the sense of Definition \ref{comp}  for the family $\{H(t)\}_{t \in [0,T]}$ and a family of linear homeomorphisms $\{\phi_{t}\}_{t \in [0,T]}$; that is, we assume $(H,\phi_{(\cdot)})$ is a compatible pair. In addition, we also assume that $(V, \phi_{(\cdot)}|_{V_0})$ is compatible. We will simply write $\phi_t$ instead of $\phi_t|_{V_0}$, and we will denote the dual operator of $\phi_t\colon V_0 \to \Vt$ by $\phi_t^*\colon \Vmt \to \Vms$; we are not interested in the dual of $\phi_t\colon \Hs \to \Ht.$ 
\end{asss}
It then follows that for each $t \in [0,T]$, $\Vt \subset \Ht$ is continuously and densely embedded. 
Let us summarise the meaning and consequences of Assumptions \ref{asss:compatibilityOfEvolvingHilbertTriple} for the convenience of the reader.
\begin{enumerate}
\item For each $t \in [0,T]$, there exists a linear homeomorphism
\[\phi_{t}\colon H_0 \to \Ht\]
such that $\phi_0$ is the identity.
\item The restriction $\phi_t|_{V_0}$ (which we will denote by $\phi_t$) is also a linear homeomorphism from $V_0$ to $\Vt$.
\item There exist constants $C_H$ and $C_V$ independent of $t \in [0,T]$ such that
\begin{align*}
\norm{\phi_t u}{\Ht} &\leq C_H\norm{u}{\Hs} &&\forall u \in \Hs, \\
\norm{\phi_t u}{\Vt} &\leq C_V\norm{u}{\Vs} &&\forall u \in \Vs.
\end{align*}
\item We will only be interested in the dual operator of $\phi_t\colon \Vs \to \Vt$, denoted by $\phi_t^*\colon \Vmt \to \Vms$, which satisfies
\begin{align*}
\norm{\phi_t^* f}{\Vms} &\leq C_V\norm{f}{\Vmt} &\forall f \in \Vmt.
\end{align*}
\item The inverses of $\phi_t$ and $\phi_t^*$ will be denoted by $\phi_{-t}$ and $\phi_{-t}^*$ respectively, and these are uniformly bounded:
\begin{align*}
\norm{\phi_{-t} u}{\Hs} &\leq \tilde C_H\norm{u}{\Ht}&&\forall u \in \Ht, \\
\norm{\phi_{-t} u}{\Vs} &\leq \tilde C_V\norm{u}{\Vt} &&\forall u \in \Vt, \\
\norm{\phi_{-t}^* f}{\Vmt} &\leq \tilde C_V\norm{f}{\Vms} &&\forall f \in \Vms.
\end{align*}
\item The maps
\begin{align*}
&t \mapsto \norm{\phi_t u}{H(t)}\qquad &&\forall u \in H_0\\
&t \mapsto \norm{\phi_t u}{V(t)}&&\forall u \in V_0
\intertext{are continuous, and the map}
&t \mapsto \norm{\phi_{-t}^* f}{V^*(t)}&&\forall f \in V^*_0
\end{align*}
is measurable.
\end{enumerate}
Our work in \S \ref{sec:assumptionsEvolution} tells us that the spaces $L^2_H$, $L^2_V,$ and $L^2_{V^*}$ are Hilbert spaces with the inner product given by the formula \eqref{eq:innerProductOnL2X}. 
\begin{remark}These homeomorphisms $\phi_t$ are similar to Arbitrary Lagrangian Eulerian (ALE) maps that are ubiquitous in applications on moving domains. See \cite{ALE} for an account of the ALE framework and a comparable set-up. 
\end{remark}
By the density of $L^2(0,T;V_0)$ in $L^2(0,T;H_0)$, we obtain the next result.
\begin{lemma}{The embedding $L^2_V \subset L^2_H$ is continuous and dense.}
\end{lemma}
Identifying $L^2_H$ with $L^2_{H^*}$ in the natural manner, we have that $L^2_V \subset L^2_H \subset L^2_{V^*}$ is a Hilbert triple. We make use of the formula
\[ \langle f, u \rangle_{L^2_{V^*}, L^2_V} = (f, u)_{L^2_H} \quad \text{whenever $f \in L^2_H$ and $u \in L^2_V$}.\]
\subsection{Abstract strong and weak material derivatives}\label{sec:abstractMaterialDerivative}
Suppose $\{\Gamma(t)\}_{t \in [0,T]}$ is a family of (sufficiently smooth) hypersurfaces evolving with velocity field $\mathbf w$, and that for each $t \in [0,T]$, $u(t)$ is a sufficiently smooth function defined on $\Gamma(t)$. Then the appropriate time derivative of $u$ \emph{takes into account the movement of the spatial points too}, and this time derivative is known as the (strong) \emph{material derivative}, which we can write informally as
\begin{equation}\label{eq:strongMaterialDerivative}
\dot u(t,x) = \frac{d}{dt}u(t,x(t)) = u_t(t,x) + \grad u(t,x)\cdot \mathbf w(t,x).
\end{equation}
This is well-studied: see \cite{gurtin} or \cite[\S 1.2]{fluids} for the flat case.
Our aim is to generalise this material derivative to arbitrary functions and arbitrary evolving spaces (and not just merely evolving surfaces).
\begin{definition}[Strong material derivative]For $\xi \in C^1_X$
define the \emph{strong material derivative} $\dot \xi \in C^0_X$ by
\begin{equation}\label{eq:defnStrongMaterialDerivative}
\dot{\xi}(t) := \phi_{t}\left(\frac{d}{dt}(\phi_{-t} \xi(t))\right).
\end{equation}
\end{definition} 
This definition is generalised from \cite{vierling}. 
So we see that the space $C^1_X$ is the space of functions with a strong material derivative, justifying the notation. In the evolving surface case, we show in \cite[\S 4]{AlpEllStiApplications} that this abstract formula agrees with \eqref{eq:strongMaterialDerivative}. The following remark observes that the pushforward of elements of $X_0$ into $X(t)$ have zero material derivative.
\begin{remark}\label{zeromatderiv}
Observe that given $\eta \in X_0$,
$$\dot{(\phi_{t}\eta)}=0$$
and that for $\xi \in C^1_X$
$$ \dot \xi= 0 \iff \exists ~ \eta \in X_0 ~~ \mbox{such that}
~~\xi(t)=\phi_{t}\eta.$$
\end{remark}
It may be the case that solutions to the PDE \eqref{eq:operatorEquationIntro}
\[\dot u(t) + A(t)u(t) = f(t)\]
may not exist if we ask for $u \in C^1_V$, that is, they may not possess strong material derivatives. We can relax this and ask for $\dot u$ to exist in a weaker sense, just like one does for the usual time derivative in parabolic problems on fixed domains. 
Heuristically, what should such a weak material derivative satisfy? Taking a clue from Lemma \ref{lem:spaceW0}, we expect 
\[\frac{d}{dt}(u(t),v(t))_{\Ht}=\langle \dot u(t), v(t)\rangle_{\Vmt, \Vt}+\langle \dot v(t), u(t) \rangle_{\Vmt,\Vt} + \text{extra term}\]
where we envisage an extra term because the Hilbert space associated with the inner product depends on $t$ itself, and certainly we should require the integration by parts formula 
\[\int_0^T \frac{d}{dt}(u(t),\eta(t))_{\Ht}=0\quad \forall \eta \in \mathcal D_V(0,T).\]
The identification of this extra term and a definition of the weak material derivative is what the rest of this section is devoted to. 
\begin{definition}[Relationship between the inner product on $\Ht$ and the space $\Hs$]\label{defn:bilinearFormb}
For all $t \in [0,T]$, define the bounded bilinear form $\hat{b}(t;\cdot,\cdot)\colon \Hs \times \Hs \to \mathbb{R}$ by
\[\hat{b}(t;u_0,v_0) = (\phi_t u_0, \phi_t v_0)_{\Ht} \quad \text{$\forall u_0$, $v_0 \in \Hs$}.\]
\end{definition}
This gives us a way of pulling back the inner product on $\Ht$ onto a bilinear form on $\Hs$ by the formula $(u,v)_{\Ht} = \hat{b}(t;\phi_{-t}u, \phi_{-t}v).$ It is also clear that
$\hat{b}(0;\cdot,\cdot) = (\cdot,\cdot)_{\Hs}$
by definition. In fact, one can see for each $t \in [0,T]$ that $\hat{b}(t;\cdot,\cdot)$ is an inner product on $\Hs$ (and it is norm-equivalent with the norm on $\Hs$); thanks to the Riesz representation theorem, there exists for each $t \in [0,T]$ a bounded linear operator $T_t\colon H_0 \to H_0$ such that
\begin{equation}
\hat b(t;u_0,v_0) = (T_tu_0, v_0)_{H_0} = (u_0, T_tv_0)_{H_0}\label{eq:assHatb}.
\end{equation}
\begin{remark}
It is not difficult to see that $T_t \equiv \phi_t^A\phi_t$, where $\phi_t^A\colon H(t) \to H_0$ denotes the Hilbert-adjoint of $\phi_t\colon H_0 \to \Ht$.
\end{remark}
\begin{asss}\label{asss:differentiabilityOfNorm}
We shall assume the following for all $u_0$, $v_0 \in H_0$:
\begin{align}
&\theta(t,u_0) := \frac{d}{dt}\norm{\phi_t u_0}{\Ht}^2 \text{ exists classically}\label{eq:assDifferentiabilityOfNorm}\\
&u_0 \mapsto \theta(t,u_0) \text{ is continuous}\label{eq:assContinuityOfDerivativeOfNorm}\\
&|\theta(t,u_0+v_0) - \theta(t,u_0-v_0)|\leq C\norm{u_0}{H_0}\norm{v_0}{H_0}\label{eq:assBoundednessOfDerivativeOfNorm}
\end{align}
where the constant $C$ is independent of $t \in [0,T]$.
\end{asss}
We are now able to define $\hat{\symbolForLittlec}(t;\cdot,\cdot)\colon \Hs\times \Hs \to \mathbb{R}$ by
\begin{align}
\hat{\symbolForLittlec}(t;u_0,v_0) &:= \frac{d}{dt}\hat{b}(t;u_0,v_0) = \frac{1}{4}\left(\theta(t,u_0+v_0) - \theta(t,u_0-v_0)\right).\label{eq:equationForHatC}
\end{align}
Denoting by $\hat{\symbolForBigC}(t)$ the operator
\begin{equation}\label{eq:assC}
\langle \hat{\symbolForBigC}(t) u_0, v_0 \rangle := \hat{\symbolForLittlec}(t;u_0, v_0),
\end{equation}
it follows by \eqref{eq:assBoundednessOfDerivativeOfNorm} that $\hat{\symbolForBigC}(t)\colon \Hs \to \Hms$.
\begin{definition}[The bilinear form $\symbolForLittlec(t;\cdot,\cdot)$]\label{defn:bilinearFormc}
For $u$, $v \in \Ht$, define the bilinear form $\symbolForLittlec(t;\cdot,\cdot)\colon \Ht \times \Ht \to \mathbb{R}$ by
\[\symbolForLittlec(t;u,v) = \hat{\symbolForLittlec}(t; \phi_{-t}u, \phi_{-t}v).\]
\end{definition}
\begin{lemma}\label{lem:bilinearFormc}
For all $u,$ $v \in L^2_H$, the map $t \mapsto \symbolForLittlec(t;u(t),v(t))$ 
is measurable and
$\symbolForLittlec(t;\cdot,\cdot)\colon\Ht \times \Ht \to \mathbb{R}$ is bounded independently of $t$:
\[|\symbolForLittlec(t;u,v)| \leq C\norm{u}{\Ht}\norm{v}{\Ht}.\]
\end{lemma}
\begin{proof}
If $u, v \in L^2_{H}$, then by \eqref{eq:equationForHatC},
\begin{align*}
\symbolForLittlec(t;u(t),v(t)) &= \hat{\symbolForLittlec}(t;\phi_{-t}u(t), \phi_{-t}v(t))\\
&= \frac{1}{4}\left(\theta(t,\phi_{-t}u(t)+\phi_{-t}v(t)) - \theta(t,\phi_{-t}u(t)-\phi_{-t}v(t))\right),
\end{align*}
and it follows that $t \mapsto \symbolForLittlec(t;u(t),v(t))$ is measurable because $t \mapsto \theta(t,\phi_{-t}w(t))$ is measurable for $w \in L^2_H$. This in turn can be seen by noticing that $\theta\colon [0,T] \times H_0 \to \mathbb{R}$ is a Carath\'eodory function: the map $t \mapsto \theta(t,x)$ is measurable and by assumption \eqref{eq:assContinuityOfDerivativeOfNorm} the map $x \mapsto \theta(t,x)$ is continuous; thus by \cite[Remark 3.4.2]{gasinski} the desired measurability is achieved. The bound on $\symbolForLittlec(t;\cdot,\cdot)$ is a consequence of the assumption \eqref{eq:assBoundednessOfDerivativeOfNorm}.
\end{proof}
\begin{lemma}\label{ass:bHatDifferentiableSecond}
For $\sigma_1$, $\sigma_2 \in C^1([0,T];H_0)$, the map $t \mapsto \hat{b}(t;\sigma_1(t), \sigma_2(t))$ 
is differentiable in the classical sense and
\[\frac{d}{dt}\hat{b}(t;\sigma_1(t), \sigma_2(t)) = \hat{b}(t;\sigma_1'(t), \sigma_2(t)) + \hat{b}(t;\sigma_1(t), \sigma_2'(t)) + \hat{\symbolForLittlec}(t;\sigma_1(t),\sigma_2(t)).\] 
\end{lemma}
This follows simply by using the definition of the derivative as a limit.
\begin{definition}[Weak material derivative]
For $u \in L^2_V$, if there exists a function $g \in L^2_{V^*}$ such that
\[\int_0^T \langle g(t), \eta(t)\rangle_{\Vmt, \Vt} =-\int_0^T (u(t), \dot{\eta}(t))_{\Ht} - \int_0^T \symbolForLittlec(t;u(t), \eta(t))\]
holds for all $\eta \in \mathcal{D}_V(0,T)$, then we say that $g$ is the \emph{weak material derivative} of $u$, and we write $$\dot u=g ~~\mbox{or} ~~\md u=g.$$
\end{definition}
This concept of a weak material derivative is indeed well-defined: if it exists, it is unique, and every strong material derivative is also a weak material derivative. It is easy to prove these facts: for uniqueness, assume there exist two material derivatives for the same function and then linearity and the density of $\mathcal{D}((0,T);V_0)$ (the space of test functions) in $L^2(0,T;V_0)$ gives the result. To show that a strong material derivative is also a weak material derivative, one can use Lemma \ref{ass:bHatDifferentiableSecond} and the relations between $\hat b(t;\cdot,\cdot)$, $\hat{\symbolForLittlec}(t;\cdot,\cdot)$, and  $\symbolForLittlec(t;\cdot,\cdot)$.
\subsection{Solution space}
We can now consider the spaces that solutions of our PDEs will lie in.
\begin{definition}[The space $W(V,V^*)$]\label{defn:solutionSpace}
Define the solution space
\begin{align*}
W(V,V^*) &= \{ u \in L^2_V \mid \dot u \in L^2_{V^*}\}
\end{align*}
and endow it with the inner product
\[(u, v)_{W(V,V^*)} = \int_0^T (u(t), v(t))_{\Vt} + \int_0^T(\dot u(t), \dot v(t) )_{\Vmt}.\]
\end{definition}

In order to prove existence theorems, we need some properties of the space $W(V,V^*)$ which turns out to be deeply linked with the following standard Sobolev--Bochner space.
\begin{definition}[The space $\mathcal W(V_0, V_0^*)$]
Define
\[\mathcal W(V_0, V_0^*) = \{v \in L^2(0,T;\Vs) \mid v' \in L^2(0,T; \Vms)\}\]
to be the space $\mathcal W(\mathcal V, \mathcal V^*)$ introduced in \S \ref{sec:1.1} with $\mathcal V = V_0$ and $\mathcal H = H_0$.
\end{definition}
It is convenient to introduce the following notion of {\it evolving space equivalence}.
\begin{defnass}\label{ass:spaceW1}
We assume that there is an \emph{evolving space equivalence} between $W(V,V^*)$ and $\mathcal W(V_0, V_0^*)$. By this we mean that
$$v \in W(V,V^*) ~~\mbox{ if and only if}~~~  \phi_{-(\cdot)} v(\cdot) \in \mathcal W(V_0, V_0^*),$$ and the equivalence of norms
\[C_1\norm{\phi_{-(\cdot)}v(\cdot)}{\mathcal W(V_0, V_0^*)} \leq \norm{v}{W(V,V^*)} \leq C_2\norm{\phi_{-(\cdot)}v(\cdot)}{\mathcal W(V_0, V_0^*)}\]
holds.
\end{defnass}
\begin{corollary}The space $W(V,V^*)$ is a Hilbert space.
\end{corollary}
We now show that Assumption \ref{ass:spaceW1} holds under certain conditions. See also the remark following the proof of the theorem.
\begin{theorem}\label{thm:spaceW1}Suppose that
\begin{equation}\label{eq:assIso}
u \in \mathcal W(V_0,V_0^*) \quad\text{if and only if}\quad T_{(\cdot)}u(\cdot) \in \mathcal W(V_0,V_0^*)\tag{T1}
\end{equation}
and that there exist operators 
\begin{align*}
\hat{S}(t)&\colon \Vms \to \Vms\qquad \text{and} \qquad \hat D(t)\colon \Vs \to \Vms
\end{align*}
such that for $u \in \mathcal{W}(V_0,V_0^*),$ 
\begin{equation}\label{eq:assDiffT}
(T_t u(t))' = \hat{S}(t)u'(t) + \hat{\symbolForBigC}(t) u(t)+\hat D(t)u(t) \tag{T2}
\end{equation}
and
\begin{align*}
\hat S(\cdot)u'(\cdot) &\in L^2(0,T;V_0^*)\qquad\text{and}\qquad\hat D(\cdot)u(\cdot) \in L^2(0,T;V_0^*).
\end{align*}
Suppose also that $\hat{S}(t)$ and $\hat{D}(t)$
are bounded independently of $t \in [0,T]$, and that $\hat{S}(t)$ has an inverse $\hat{S}(t)^{-1}\colon \Vms \to \Vms$ which also is bounded independently of $t \in [0,T]$. Then  $W(V,V^*)$ is equivalent to $\mathcal W(V_0, V_0^*)$ in the sense of Definition \ref{ass:spaceW1}.
\end{theorem}
\begin{proof}
First, suppose $u \in \mathcal W(V_0,V_0^*)$. Clearly $\phi_{(\cdot)}u(\cdot) \in L^2_V$ and we need only to show that $\md (\phi_{(\cdot)}u(\cdot)) \in L^2_{V^*}$ exists. Let $\eta \in \mathcal{D}_V(0,T)$ and consider 
\begin{align}
\nonumber \int_0^T (\phi_t u(t), \dot \eta (t))_{\Ht} 
\nonumber &= \int_0^T (T_tu(t), (\phi_{-t}\eta(t))')_{\Hs}\tag{rewriting the integrand using $\hat b(t;\cdot,\cdot)$ and \eqref{eq:assHatb}}\\
\nonumber &= -\int_0^T \langle \hat{S}(t)u'(t) + \hat{\symbolForBigC}(t) u(t) + \hat D(t)u(t), \phi_{-t}\eta(t)\rangle_{\Vms,\Vs}\tag{by \eqref{eq:assIso} and \eqref{eq:assDiffT}}\\
\nonumber &= -\int_0^T \langle \phi_{-t}^*(\hat{S}(t)u'(t)+\hat D(t)u(t)), \eta(t)\rangle_{\Vmt,\Vt}\\
&\quad -\int_0^T \symbolForLittlec(t;\phi_tu(t),\eta(t)).\label{eq:proofIsomorphism1}
\end{align}
This shows that $\md(\phi_{(\cdot)}u(\cdot))$ exists.

Conversely, let $u \in W(V,V^*).$ We need to show the existence of $(\phi_{-(\cdot)}u(\cdot))'$ in $L^2(0,T;V_0^*).$ We start with the weak material derivative condition:
\[\int_0^T  \langle \dot u(t), \eta (t) \rangle_{\Vmt, \Vt} = -\int_0^T (u(t), \dot \eta (t))_{\Ht} - \int_0^T \symbolForLittlec(t;u(t),\eta(t))\]
for test functions $\eta \in \mathcal D_V(0,T)$. Pulling back leads to
\begin{align*}
\int_0^T  \langle \phi_t^*\dot u(t), \phi_{-t}\eta (t) \rangle_{\Vms, \Vs} &= -\int_0^T \hat b(t;\phi_{-t}u(t), (\phi_{-t}\eta (t))')\\
&\quad + \int_0^T\hat{\symbolForLittlec}(t;\phi_{-t}u(t),\phi_{-t}\eta(t)).
\end{align*}
Using \eqref{eq:assHatb} and \eqref{eq:assC} and rearranging:
\begin{align}
\int_0^T (T_t\phi_{-t}u(t), (\phi_{-t}\eta (t))')_{H_0} &=  -\int_0^T  \langle \phi_t^*\dot u(t)+\hat{\symbolForBigC}(t)\phi_{-t}u(t), \phi_{-t}\eta (t) \rangle_{\Vms, \Vs}\label{eq:proofIsomorphism2}.
\end{align}
It follows that $T_{(\cdot)}\phi_{-(\cdot)}u(\cdot)$ has a weak derivative, and hence by \eqref{eq:assIso} as does $\phi_{-(\cdot)}u(\cdot)$. This proves the bijection between $\mathcal{W}(V_0,V_0^*)$ and $W(V,V^*)$.

For the equivalence of norms,
let $u \in W(V,V^*).$ From \eqref{eq:proofIsomorphism1}, we see that
\[\dot u(t) = \phi_{-t}^*(\hat{S}(t)(\phi_{-t}u(t))' + \hat D(t)\phi_{-t}u(t))\]
which we can bound thanks to the boundedness of $\hat{S}(t)$ and $\hat D(t)$:
\[\norm{\dot u(t)}{\Vt} \leq C\left(\norm{(\phi_{-t}u(t))'}{\Vms} + \norm{\phi_{-t}u(t)}{\Vs}\right).\]
So we have achieved $\norm{u}{W(V,V^*)} \leq C_2\norm{\phi_{-(\cdot)}u(\cdot)}{\mathcal W(V_0, V_0^*)}.$ {For the reverse inequality, we use \eqref{eq:assDiffT} and \eqref{eq:proofIsomorphism2} to find}
\[ (\phi_{-t}u(t))'= \hat{S}(t)^{-1}(\phi_t^*\dot u(t)-\hat D(t)\phi_{-t}u(t)).\]
From this we obtain a bound of the form
\[\norm{(\phi_{-t}u(t))'}{\Vms} \leq C\left(\norm{\dot u(t)}{\Vmt} + \norm{u(t)}{\Vt}\right)\]
which implies the result.
\end{proof}
\begin{remark}If we knew that $T_tv_0 \in V_0$ for every $v_0 \in V_0$, then the assumption \eqref{eq:assDiffT} would follow from \eqref{eq:assIso} with $\langle \hat S(t)f, v \rangle_{\Vms, \Vs} := \langle f, T_tv \rangle_{\Vms, \Vs}$ and $\hat D(t) \equiv 0$.
\end{remark}
We are able to specify initial conditions of solutions to PDEs via the following lemma, which is an easy consequence of the continuity of the embedding $\mathcal W(V_0,V_0^*) \subset C^0([0,T];H_0)$.
\begin{lemma}\label{lem:initialConditionsWellDefined}The embedding $W(V,V^*) \subset C^0_H$ holds, hence for any $u \in W(V,V^*)$ the evaluation $t \mapsto u(t)$ is well-defined for every $t \in [0,T]$. Furthermore, we have the inequality 
\[\max_{t \in [0,T]}\norm{u(t)}{H(t)} \leq C\norm{u}{W(V,V^*)}\qquad\text{$\forall u \in W(V,V^*)$}.\]
\end{lemma}
This lemma allows us to define the subspace
\begin{align*}
W_0(V,V^*) &= \{ u \in W(V,V^*) \mid u(0) = 0 \}.
\end{align*}
\begin{definition}[The space $W(V,H)$]Define the space
\begin{align*}
W(V,H) &= \{ u \in L^2_V \mid \dot u \in L^2_{H}\}.
\end{align*}
\end{definition}
In order to obtain a regularity result, we need to make the following natural assumption, which will also tell us that $W(V,H)$ is a Hilbert space.
\begin{ass}\label{ass:evolvingSpaceEquivalenceStronger}
We assume that there is an evolving space equivalence between $W(V,H)$ and $\mathcal W(V_0,H_0)$.
\end{ass}
Let us note that this assumption follows if, for example, the assumption \eqref{eq:assIso} is changed in the natural way and the maps $\hat S(t)$ and $\hat D(t)$ of Theorem \ref{thm:spaceW1} satisfy $\hat S(t)\colon H_0 \to H_0$ and $\hat D(t)\colon V_0 \to H_0$, with both maps and $\hat S(t)^{-1}$ being bounded independently of $t \in [0,T]$, and if $\hat S(\cdot) u'(\cdot)$, $\hat D(\cdot)u(\cdot) \in L^2(0,T;H_0)$ for $u \in \mathcal{W}(V_0,H_0)$.

\paragraph{Some density results}With the help of the density result in Lemma \ref{lem:spaceW0}, it is easy to prove the following lemma.
\begin{lemma}\label{lem:densityOfDVinW}The space $\mathcal{D}_V[0,T]$ in dense in $W(V,V^*)$.
\end{lemma}
The next few results are necessary to prove Lemma \ref{lem:temamType}, which turns out to be vital for one of our existence proofs.
\begin{lemma}\label{cor:densityDV}For every $\eta \in \mathcal{D}_V(0,T)$, there exists a sequence $\{\eta_n\} \subset \mathcal{D}_V(0,T)$
of the form
\[\eta_n(t) = \sum_{j=1}^n \zeta_j(t) \phi_t w_j\qquad \text{where $\zeta_j \in \mathcal{D}(0,T)$ and $w_j \in V_0$,}\]
such that $\eta_n \to \eta$ in $W(V,V^*)$.
\end{lemma}
\begin{proof}
It suffices to show that for every $\psi \in \mathcal{D}((0,T);V_0)$, there exists a sequence $\{\psi_n\} \subset \mathcal{D}((0,T);V_0)$ of the form
\[\psi_n(t) = \sum_{j=1}^n \zeta_j(t) w_j\qquad \text{where $\zeta_j \in \mathcal{D}(0,T)$ and $w_j \in V_0$,}\]
such that $\psi_n \to \psi$ in $\mathcal W(V_0,V_0^*)$.

Let $w_j$ be an orthonormal basis for $V_0$. Given $\psi \in \mathcal{D}((0,T);V_0)$, define 
\[\psi_n(t) = \sum_{j=1}^n (\psi(t), w_j)_{V_0}w_j,\]
i.e., $\zeta_j(t) = (\psi(t), w_j)_{V_0}.$ It is clear that $\zeta_j$ vanishes at the boundary (since $\psi$ does), and 
$\zeta_j^{(m)}(t) = (\psi^{(m)}(t), w_j)_{V_0}$
also implies that $\zeta_j \in \mathcal{D}(0,T)$. What remains to be checked is that $\psi_n \to \psi$ in $\mathcal W(V_0, V_0^*)$. We have the pointwise convergence $\psi_n(t) \to \psi(t)$ in $V_0$ because $w_j$ is a basis, and there is also the uniform bound $\norm{\psi_n(t)}{V_0} \leq \norm{\psi(t)}{V_0}$. So by the dominated convergence theorem,
\[\psi_n \to \psi \qquad \text{in $L^2(0,T;V_0)$}.\]
The same reasoning applied to $\psi_n'$ allows us to conclude.
\end{proof}

\paragraph{Transport theorem}
Like in part (3) of Lemma \ref{lem:spaceW0}, we want to differentiate the inner product on $\Ht$. Writing Lemma \ref{ass:bHatDifferentiableSecond} in different notation, we obtain for $u$, $v \in C^1_H$ the transport theorem for $C^1_H$ functions:
\[\frac{d}{dt}(u(t),v(t))_{\Ht}=(\dot u(t), v(t))_{\Ht} + (u(t), \dot v(t))_{\Ht} + \symbolForLittlec(t;u(t),v(t)).\]
We can obtain a formula for general functions $u$, $v \in W(V,V^*)$ by means of a density argument.
\begin{theorem}[Transport theorem]\label{thm:transportTheorem}
For all $u$, $v \in W(V,V^*)$, the map
\[t \mapsto (u(t),v(t))_{\Ht}\]
is absolutely continuous on $[0,T]$ and 
\[\frac{d}{dt}(u(t),v(t))_{\Ht} = \langle \dot u(t), v(t) \rangle_{\Vmt, \Vt}+\langle \dot v(t), u(t) \rangle_{\Vmt, \Vt} + \symbolForLittlec(t;u(t),v(t))\]
for almost every $t \in [0,T]$.
\end{theorem}
\begin{proof}
Given $u \in W(V,V^*)$, by Lemma \ref{lem:densityOfDVinW}, there exists a sequence $u_m \in \mathcal{D}_V[0,T]$ converging to $u$ in $W(V,V^*)$. 
By the transport theorem for $C^1_H$ functions, the $u_m$ satisfy
\begin{equation*}
\frac{d}{dt}\norm{u_m(t)}{\Ht}^2 = 2(\dot u_m(t), u_m(t))_{H(t)} + \symbolForLittlec(t;u_m(t),u_m(t)).
\end{equation*} 
This statement written in terms of weak derivatives is that for any $\zeta \in \mathcal{D}(0,T)$, it holds that
\begin{align}
\nonumber -\int_0^T &\norm{u_m(t)}{\Ht}^2\zeta'(t)\\
 &= \int_0^T \left(2\langle \dot u_m(t), u_m(t)\rangle_{\Vmt, \Vt} + \symbolForLittlec(t;u_m(t),u_m(t))\right)\zeta(t).\label{eq:proofTransportTheorem2}
\end{align}
Now we must pass to the limit in this equation. For the left hand side, because $u_m \to u$ in $L^2_H$, we have by the reverse triangle inequality
\[\int_0^T \big|\norm{u_m(t)}{\Ht} - \norm{u(t)}{\Ht}\big|^2 \leq \int_0^T \norm{u_m(t) - u(t)}{\Ht}^2 \to 0,\]
i.e.,
$\norm{u_m(\cdot)}{H(\cdot)} \to \norm{u(\cdot)}{H(\cdot)}$ in $L^2(0,T)$, which implies that
\[\norm{u_m(\cdot)}{H(\cdot)}^2 \to \norm{u(\cdot)}{H(\cdot)}^2\qquad\text{in $L^1(0,T)$.}\] 
Clearly, the functional $F\colon L^1(0,T) \to \mathbb{R}$, defined
\[F(y) = \int_0^T y(t)\zeta'(t),\]
is an element of $L^1(0,T)^*$ because $\zeta'(t)$ is bounded. Therefore, we have convergence of the left hand side of \eqref{eq:proofTransportTheorem2}:
\[-\int_0^T \norm{u_m(t)}{\Ht}^2\zeta'(t) \to -\int_0^T \norm{u(t)}{\Ht}^2\zeta'(t).\]
To deal with the terms on the right hand side of \eqref{eq:proofTransportTheorem2}, we require the estimates
\begin{align*}
&|\langle \dot u_m(t), u_m(t) \rangle_{\Vmt, \Vt} -\langle \dot u(t), u(t) \rangle_{\Vmt, \Vt}|\\
&\leq \norm{\dot u_m(t)}{\Vmt}\norm{u_m(t) - u(t)}{\Vt} + \norm{\dot u_m(t) - \dot u(t)}{\Vmt}\norm{u(t)}{\Vt}\\
\intertext{and}
&|\symbolForLittlec(t;u_m(t),u_m(t)) - \symbolForLittlec(t;u(t),u(t))|\\
&\quad \leq C_1\left(\norm{u_m(t)}{\Ht}\norm{u_m(t)-u(t)}{\Ht}+\norm{u_m(t)-u(t)}{\Ht}\norm{u(t)}{\Ht}\right).
\end{align*}
With these, it is easy to show that
\begin{align*}
&\bigg|\int_0^T \left(2\langle \dot u_m(t), u_m(t)\rangle_{\Vmt, \Vt} + \symbolForLittlec(t;u_m(t),u_m(t))\right)\zeta(t)\\
&\qquad  \qquad \qquad - \int_0^T \left(2\langle \dot u(t), u(t)\rangle_{\Vmt, \Vt} + \symbolForLittlec(t;u(t),u(t))\right)\zeta(t)\bigg|
\to 0.
\end{align*}
In other words, as $m \to \infty$, the equation \eqref{eq:proofTransportTheorem2} becomes
\begin{equation}\label{eq:proofTransportTheorem4}
-\int_0^T \norm{u(t)}{\Ht}^2\zeta'(t)=\int_0^T \left(2\langle \dot u(t), u(t)\rangle_{\Vmt, \Vt} + \symbolForLittlec(t;u(t),u(t))\right)\zeta(t),
\end{equation}
which is precisely the statement
\[\frac{d}{dt}\norm{u(t)}{\Ht}^2 = 2\langle \dot u(t), u(t)\rangle_{\Vmt, \Vt} + \symbolForLittlec(t;u(t),u(t))\]
in the sense of distributions.
From this, it follows that
\begin{align}
\frac{d}{dt}(u(t),v(t))_{\Ht} &= \langle \dot u(t), v(t) \rangle_{\Vmt, \Vt}+\langle \dot v(t), u(t) \rangle_{\Vmt, \Vt}+ \symbolForLittlec(t;u(t),v(t))\label{eq:proofTransportTheorem5}
\end{align}
holds in the weak sense. So we have shown the transport theorem in the weak sense. However, because the right hand side of the above is in $L^1(0,T)$ (since the right hand side of \eqref{eq:proofTransportTheorem4} holds for every $\zeta \in \mathcal{D}(0,T)$) and because $(u(t),v(t))_{\Ht} \in L^1(0,T)$, it follows that $(u(t),v(t))_{\Ht}$ is a.e. equal to an absolutely continuous function, with (classical) derivative a.e., and therefore \eqref{eq:proofTransportTheorem5} exists in the classical sense.
\end{proof}
We shall use the following corollary frequently without referencing in future sections.
\begin{corollary}[Integration by parts]
For all $u$, $v \in W(V,V^*)$, the integration by parts formula
\begin{align*}
(u(T)&,v(T))_{H(T)} - (u(0),v(0))_{H_0} \\
&= \int_0^T\langle \dot u(t), v(t) \rangle_{\Vmt, \Vt}+ \langle \dot v(t), u(t) \rangle_{\Vmt, \Vt}
+ \symbolForLittlec(t;u(t),v(t))\;\mathrm{d}t
\end{align*}
holds.
\end{corollary}

\section{Formulation of the problem and statement of results}\label{sec:main}

\subsection{Precise formulation of the PDE}

Having built up the essential function spaces and results, we are now in a position to formulate PDEs on evolving spaces. We continue with the framework and notation of \S \ref{sec:function_spaces}; we reiterate in particular Assumptions \ref{asss:compatibilityOfEvolvingHilbertTriple}, \ref{asss:differentiabilityOfNorm}, and \ref{ass:spaceW1} (which relate respectively to the compatibility of the evolving Hilbert spaces, a well-defined material derivative, and the evolving space equivalence). We are interested in the existence and uniqueness of solutions $u \in W(V,V^*)$ to equations of the form
\begin{equation}\label{eq:operatorEquation}
\begin{aligned}
L \dot u + A u+ \symbolForBigC u &= f &&\text{in $L^2_{V^*}$}\\
u(0) &= u_0 &&\text{in $H_0$},
\end{aligned}\tag{$\textbf{P}$}
\end{equation}
where we identify 
\begin{equation*}\label{eq:identificationOfOperators}
\begin{aligned}
(L\dot u)(t) &= L(t)\dot u(t)\\
(Au)(t) &= A(t)u(t)\\
(\symbolForBigC u)(t) &= \symbolForBigC(t)u(t),
\end{aligned}
\end{equation*}
with $L(t)$ and $A(t)$ being linear operators that satisfy the minimal assumptions given below, and 
\begin{align*}
&\symbolForBigC(t)\colon\Ht \to \Hmt \quad \text{is defined by} \quad \langle \symbolForBigC(t) v,w  \rangle_{\Hmt, \Ht} := \symbolForLittlec(t;v,w),\end{align*}
with $\symbolForLittlec(t;\cdot,\cdot)$ the bilinear form in the definition of the weak material derivative (Definition \ref{defn:bilinearFormc}). Note that $\symbolForBigC(t)$ is symmetric in the sense that
$ \langle \symbolForBigC(t) v, w \rangle= \langle
\symbolForBigC(t) w, v \rangle.$
\begin{remark}We showed in Lemma \ref{lem:initialConditionsWellDefined} that specifying the initial condition as in \eqref{eq:operatorEquation} is well-defined.
\end{remark} 
\begin{asss}[Assumptions on $L(t)$]\label{asss:onL}
In the following, all constants $C_i$ are positive and independent of $t \in [0,T]$. 

We shall assume that for all $g \in L^2_{V^*}$,
\begin{align}
Lg \in L^2_{V^*}\qquad\text{and}\qquad C_1 \norm{g}{L^2_{V^*}} \leq \norm{Lg}{L^2_{V^*}} \leq C_2 \norm{g}{L^2_{V^*}}\label{eq:assLvstarInL2Vstar}\tag{L1}.
\end{align}
We suppose that the restriction $L|_{L^2_H}$  satisfies $L|_{L^2_H}\colon L^2_H \to L^2_H$, we identify $(L|_{L^2_H}h)(t) =: L_H(t)h(t),$ and we suppose that 
\begin{align*}
L_H(t)&\colon \Ht \to \Ht \text{ is symmetric, and}\\
L_H(t)&\colon \Vt \to \Vt.
\end{align*}
We simply write $L$ and $L(t)$ for the above restrictions. Furthermore, for almost every $t \in [0,T]$, we assume
\begin{align}
\langle L(t)g,v\rangle_{\Vmt, \Vt} &= \langle g,L(t)v\rangle_{\Vmt,\Vt}&&\text{$\forall g \in \Vmt$, $\forall v \in \Vt$}\tag{L2}\label{eq:assSymmetricityOfL}\\
\norm{L(t)h}{\Ht} &\leq C_3 \norm{h}{\Ht}&&\text{$\forall h \in \Ht$}\label{eq:assBoundednessOfl}\tag{L3}\\
(L(t)h, h)_{\Ht} &\geq C_4\norm{h}{\Ht}^2&&\text{$\forall h \in \Ht$}\label{eq:assCoercivityOfl}\tag{L4}\\
Lv &\in L^2_V&&\text{$\forall v \in L^2_V$}\label{eq:assLvInL2V}\tag{L5}\\
v \in W(V,V^*) &\iff Lv \in W(V,V^*)\label{eq:assLvInW1}\tag{L6},\\
\intertext{and we suppose the existence of a (linear symmetric) map $\dot{{L}}\colon L^2_V \to L^2_{V^*}$  (and we identify $(\dot{L}v)(t) =: \dot{L}(t)v(t)$) satisfying}
\md (Lv) &= \dot{L} v + L \dot v \in L^2_{V^*} &&\text{$\forall v\in W(V,V^*)$}\label{eq:assProductRule}\tag{L7}\\
\lVert{\dot{L}(t)v}\rVert_{\Vmt} &\leq C_5\norm{v}{\Ht}&&\text{$\forall v \in \Vt$}\label{eq:assBoundednessOfdotL}\tag{L8}.
\end{align}
\end{asss}
\begin{asss}[Assumptions on $A(t)$]\label{asss:aWithoutL}
Suppose that the map
\begin{align}
\nonumber t &\mapsto \langle A(t)v(t),w(t)\rangle_{\Vmt, \Vt} &&\forall v, w \in L^2_V
\intertext{is measurable, and that there exist positive constants $C_1$, $C_2$ and $C_3$ independent of $t$ such that the following holds for almost every $t \in [0,T]$:}
\langle A(t)v,v\rangle_{\Vmt, \Vt} &\geq C_1 \norm{v}{\Vt}^2-C_2 \norm{v}{\Ht}^2&&\forall v \in \Vt\tag{A1}\label{eq:assCoercivityOfa}\\
|\langle A(t)v,w\rangle_{\Vmt, \Vt}| &\leq C_3\norm{v}{\Vt}\norm{w}{\Vt}&&\forall v, w \in \Vt\tag{A2}\label{eq:assBoundednessOfa}.
\end{align}
\end{asss}
Observe that we have generalised the PDE \eqref{eq:operatorEquationIntro} by introducing the operator $L.$  The standard equation
\[\dot u + A u + \symbolForBigC u = f\]
is a special case of \eqref{eq:operatorEquation} when $L = \text{Id}$. Our demands in Assumptions \ref{asss:onL} are (of course) automatically met in this case. Also, there is no loss of generality by considering the equation \eqref{eq:operatorEquation} instead of the more natural equation $L\dot u + Au = f.$ We include the operator $\symbolForBigC$ purely because it is convenient in applications (such as those in \cite{AlpEllStiApplications}).

Implicit in \eqref{eq:operatorEquation} is the claim that $A u$ and $\symbolForBigC u$ are elements of $L^2_{V^*}.$ The fact $A u \in L^2_{V^*}$ follows by the weak (and thus strong) measurability of $t \mapsto \phi_{t}^*A(t)u(t)$ and the boundedness of $A(t)$, and similarly one obtains the result $\symbolForBigC u \in L^2_{V^*}.$ 
Let us mention an important consequence of the transport theorem (Theorem \ref{thm:transportTheorem}) and assumptions \eqref{eq:assSymmetricityOfL}, \eqref{eq:assLvInW1} and \eqref{eq:assProductRule}.
\begin{lemma}\label{lem:derivativeOfL}{For every $v,$ $w \in W(V,V^*)$, the map $t \mapsto (L(t)v(t),w(t))_{\Ht}$ is absolutely continuous with derivative}
\begin{align}
\nonumber \frac{d}{dt}(L(t)v(t),w(t))_{\Ht} &= \langle L(t)\dot v(t), w(t)\rangle_{\Vmt, \Vt} + \langle L(t)\dot w(t), v(t)\rangle_{\Vmt, \Vt} \\
&\quad+ \langle M(t)v(t),w(t)\rangle_{\Vmt,\Vt}\label{eq:assDifferentiabilityOfl}
\end{align}
almost everywhere, where $M(t)\colon \Vt \to \Vmt$ is the operator
\begin{align*}
\langle M(t)v, w \rangle_{\Vmt,\Vt} &:= \langle \dot{L}(t)v, w \rangle_{\Vmt, \Vt} + \langle \symbolForBigC(t) L(t) v, w \rangle_{\Vmt, \Vt}
\end{align*}
which generates the bounded bilinear form $m(t;\cdot,\cdot)\colon \Vt \times \Vt \to \mathbb{R}$:
\[m(t;v,w):=\langle M(t)v, w\rangle_{\Vmt, \Vt}.\]
\end{lemma}

To conclude this preliminary subsection we state and prove the following  lemma which is used in \S \ref{sec:proofExistenceSecondGalerkin}.

\begin{lemma}\label{lem:temamType}
Let $u \in L^2_V$ and $g \in L^2_{V^*}.$ Then
\[\dot u \in L^2_{V^*} \text{ exists }\text{and } L\dot u = g \qquad\]
if and only if
\begin{equation}\label{eq:temamType}
\frac{d}{dt}(L(t)u(t),\phi_t v_0)_{\Ht} = \langle g(t)+M(t)u(t), \phi_tv_0 \rangle_{\Vmt, \Vt} ~~ \mbox{for all}~~ v_0 \in V_0
\end{equation}
in the weak sense.
\end{lemma}

\begin{proof}[Proof of Lemma \ref{lem:temamType}]
{If $u \in W(V,V^*)$ and $L\dot u = g$, then \eqref{eq:temamType} follows easily by utilising $\md(\phi_t v_0) = 0$ and the previous lemma.} For the converse, first, we see from Lemma \ref{cor:densityDV} that given any $\eta \in \mathcal{D}_V(0,T)$, there exist functions $\eta_n \in \mathcal{D}_V(0,T)$ of the form
\[\eta_n(t)=\sum_j \zeta_j(t) \phi_tw_{j}\]
with $\zeta_j \in \mathcal{D}(0,T)$ and $w_{j} \in V_0$ such that $\norm{\eta - \eta_n}{W(V,V^*)} \to 0.$ 
Now, \eqref{eq:temamType} states that
\begin{align*}
\int_0^T (L(t)u(t), \zeta'(t)\phi_tv_0)_{\Ht} 
&= -\int_0^T \langle g(t)+{M}(t)u(t), \zeta(t)\phi_tv_0\rangle_{\Vmt, \Vt}
\end{align*}
holds for all $\zeta \in \mathcal{D}(0,T)$ and all $v_0 \in V_0.$ In particular, we may pick $\zeta = \zeta_j$ and $v_0 = w_{j}$ and sum up over $j$ to obtain
\[\int_0^T (L(t)u(t), \dot\eta_n(t))_{\Ht} 
= -\int_0^T \langle g(t)+{M}(t)u(t), \eta_n(t)\rangle_{\Vmt, \Vt}.\]
Passing to the limit and using the convergence above, we find
\begin{align*}
\int_0^T (L(t)u(t), \dot\eta(t))_{\Ht} 
&= -\int_0^T \langle g(t)+{M}(t)u(t), \eta(t)\rangle_{\Vmt, \Vt}\\
&= -\int_0^T \langle g(t)+\dot{L}(t)u(t) + \symbolForBigC(t)L(t)u(t), \eta(t)\rangle_{\Vmt, \Vt}
\end{align*}
for arbitrary $\eta\in \mathcal D_V(0,T)$, i.e., we have the existence of $\md(Lu) = g + \dot{L}u \in L^2_{V^*}$ which, thanks to assumptions \eqref{eq:assLvInW1} and \eqref{eq:assProductRule} implies that $L\dot u = g$.
\end{proof}
\subsection{Well-posedness and regularity}
We begin with a well-posedness theorem which is proved in \S \ref{sec:proofs}. A sketch of a second proof will be presented in \S \ref{sec:proofExistenceSecondGalerkin} where we utilise a Galerkin method.
\begin{theorem}[Well-posedness of \eqref{eq:operatorEquation}]\label{thm:existenceWithoutL}
Under the assumptions in Assumptions \ref{asss:onL} and \ref{asss:aWithoutL}, for $f \in L^2_{V^*}$ and $u_0 \in H_0$, there is a unique solution $u \in W(V,V^*)$ satisfying \eqref{eq:operatorEquation} such that
\begin{equation*}
\norm{u}{W(V,V^*)} \leq C\left(\norm{u_0}{H_0}+\norm{f}{L^2_{V^*}}\right).
\end{equation*}
\end{theorem}
Now, suppose we now know that $f \in L^2_H$ and $u_0 \in V_0$. Can we expect the same regularity on the solution $u$ as holds in the case of stationary spaces? It turns out that we can obtain $\dot u \in L^2_H$ under some additional assumptions, including some on the differentiability of $A(t)$. 

Before we list these assumptions, let us just note that if we define bilinear forms $l(t;\cdot,\cdot)\colon \Vmt \times \Vt \to \mathbb{R}$ and $a(t;\cdot,\cdot)\colon \Vt \times \Vt \to \mathbb{R}$ to satisfy
\begin{align*}
 l(t;g,w) &:= \langle L(t) g,w \rangle_{\Vmt,\Vt}\\
a(t;v,w) &:=\langle A(t) v,w  \rangle_{\Vmt, \Vt},
\end{align*}
then the problem \eqref{eq:operatorEquation} is in fact equivalent to
\begin{equation}\label{eq:defnWeakSolutionSecond}
\begin{aligned}
l(t; \dot u(t), v) + a(t;u(t),v) + \symbolForLittlec(t;u(t),v) &= \langle f(t), v \rangle_{\Vmt, \Vt}\\
u(0) &= u_0
\end{aligned}
\end{equation}
for all $v \in V(t)$ and for almost every $t \in [0,T]$ (the null set is independent of $v$). Similarly, if $f \in L^2_H$ and $\dot u \in L^2_H$, then \eqref{eq:operatorEquation} is equivalent to
\begin{equation}\label{eq:defnWeakSolutionWithoutL}
\begin{aligned}
l(t; \dot u(t), v) + a(t;u(t),v) + \symbolForLittlec(t;u(t),v) &= (f(t), v)_{H(t)}\\
u(0) &= u_0
\end{aligned}\tag{$\textbf{P'}$}
\end{equation}
for all $v \in V(t)$ and for almost every $t \in [0,T]$, where now $l(t;\cdot,\cdot)\colon \Ht \times \Ht \to \mathbb{R}$ is $l(t;\cdot,\cdot) = (L(t)\cdot, \cdot)_{\Ht}$. It is this form of the problem that turns out to be more convenient to work with to show regularity. 
To see the equivalence, for one side, we may take the duality pairing of \eqref{eq:operatorEquation} with $v = \xi\phi_{(\cdot)}v_0$ where $v_0 \in V_0$ and $\xi \in \mathcal{D}(0,T)$; then an argument involving the separability of $V_0$ gives \eqref{eq:defnWeakSolutionWithoutL}. The converse follows by the density of simple measurable functions in $L^2_V$ (see Lemma \ref{cor:convergenceSimpleMeasurableFunctions}).

Since $V_0$ is separable, we may find a basis $\{\chi_j^0\}$, by which we mean that for all $N \in \mathbb{N}$, the set $\{\chi_j^0\}_{j=1}^N$ is linearly independent and finite linear combinations of $\chi_j^0$ are dense in $V_0$. 
\begin{ass}\label{ass:basisFunctionsForRegularity}
We assume that there exists a basis $\{\chi_j^0\}_{j \in \mathbb{N}}$ of $V_0$ and a sequence $\{u_{0N}\}_{N \in \mathbb{N}}$ with $u_{0N} \in \text{span}\{\chi_1^0, ..., \chi_N^0\}$ for each $N$, such that
\begin{align}
u_{0N} &\to u_0 \qquad \text{in $V_0$}\label{eq:assu0NConvergesInV0}\tag{B1}\\
\norm{u_{0N}}{H_0} &\leq C_1\norm{u_0}{H_0}\label{eq:assu0NBoundedInH0}\tag{B2}\\
\norm{u_{0N}}{V_0} &\leq C_2\norm{u_0}{V_0}\label{eq:assInitialDataApproximationBoundedInV}\tag{B3}
\end{align}
{where $C_1$ and $C_2$ do not depend on $N$ or $u_0$.} 
\end{ass}
\begin{remark}\label{rem:onBasisFunctions}
{Such a basis as required by the last assumption always exists if $V_0 \subset H_0$ is compact thanks to Hilbert--Schmidt theory. In fact, in such a case we can find a basis $\chi_j^0$ which is orthonormal in $H_0$ and orthogonal in $V_0$.}
\end{remark}
Let $AC([0,T])$ be the space of absolutely continuous functions from $[0,T]$ into $\mathbb{R}$.
\begin{definition}\label{defn:galerkinSpace}
{We define the space
\[\tilde{C}^1_V = \{u \mid u(t) = \sum_{j=1}^m \alpha_j(t)\chi_j^t, \text{ $m \in \mathbb{N}$, $\alpha_j \in AC([0,T])$ and $\alpha_j' \in L^2(0,T)$}\}.\]
Note that $\tilde C^1_V \subset C^0_V$ and $\tilde C^1_V \subset W(V,V)$.}
\end{definition}
\begin{remark}\label{rem:materialDerivativeOfTildeC1VFunctions}
{Note that if $u \in \tilde C^1_V$ with $u(t) = \sum_{j=1}^m \alpha_j(t)\chi_j^t$ as in the definition then $\dot u(t) = \sum_{j=1}^m \alpha'_j(t)\chi_j^t.$ We skip the proof which is straightforward: just use the definition of the weak material derivative and perform some manipulations. We could not have calculated the strong material derivative of $u$ via the 
formula \eqref{eq:defnStrongMaterialDerivative} because the pullback \[\phi_{-(\cdot)}u(\cdot) = \sum_{j=1}^n \alpha_j(\cdot)\chi_j^0\]
is not necessarily in $C^1([0,T];V_0)$ since the $\alpha_j$ are not necessarily $C^1$.}
\end{remark}
\begin{asss}[Further assumptions on $a(t;\cdot,\cdot)$]\label{asss:aAndasWithL}
Suppose that $a(t;\cdot,\cdot)$ has the form
\[a(t;\cdot,\cdot) = a_s(t;\cdot,\cdot) + a_n(t;\cdot,\cdot)\]
where 
\begin{align*}
&a_s(t;\cdot, \cdot)\colon \Vt \times \Vt \to \mathbb{R}\\
&a_n(t;\cdot, \cdot)\colon\Vt \times \Ht \to \mathbb{R}
\end{align*}
are bilinear forms (we allow the possibility $a_n \equiv 0$) such that {the map} 
\begin{equation}\label{eq:assAbsCtyOfas}
t \mapsto a_s(t;y(t),y(t)) \text{ {is absolutely continuous on $[0,T]$ for all $y \in \tilde C^1_V$.}}\tag{A3}
\end{equation}
Suppose also that there exist positive constants $C_1$, $C_2$ and $C_3$ independent of $t$ such that for almost every $t \in [0,T]$, 
\begin{align}
|a_n(t;v,w)| &\leq C_1\norm{v}{\Vt}\norm{w}{\Ht}&&\forall v \in \Vt, w \in \Ht
\tag{A4}\label{eq:assBoundednessOfan}\\
|a_s(t;v,w)| &\leq C_2\norm{v}{\Vt}\norm{w}{\Vt}&&\forall v, w \in \Vt\tag{A5}\label{eq:assBoundednessOfas}\\
a_s(t;v,v) &\geq 0\tag{A6}&&\forall v \in \Vt \label{eq:assPositivityOfas}\\
\frac{d}{dt}a_s(t;y(t),y(t)) &= 2a_s(t;y(t),\dot y(t)) + r(t;y(t))&&\forall y \in \tilde{C}^1_V, \tag{A7}\label{eq:assDifferentiabilityOfas}\\
\intertext{where the $\frac{d}{dt}$ here is the classical derivative, and $r(t;\cdot)\colon\Vt \to \mathbb{R}$ satisfies}
|r(t;v)| &\leq C_3\norm{v}{\Vt}^2&&\forall v \in \Vt\tag{A8}\label{eq:assBoundednessOfF}.
\end{align}
\end{asss}
\begin{remark}\label{rem:asssOna}
Note that we require only one part of the bilinear form $a(t;\cdot,\cdot)$ to be differentiable; however, any potentially non-differentiable terms require the stronger boundedness condition \eqref{eq:assBoundednessOfan}.
\end{remark}
As alluded to above, it is permissible to take $a_n \equiv 0$ so that $a \equiv a_s$. In this case, we are in the same situation as in Assumptions \ref{asss:aWithoutL} except with the addition of \eqref{eq:assAbsCtyOfas}, \eqref{eq:assPositivityOfas}, \eqref{eq:assDifferentiabilityOfas}, and \eqref{eq:assBoundednessOfF}.

We have the following regularity result proved in  \S \ref{sec:galerkinApproximations}.

\begin{theorem}[Regularity of the solution to \eqref{eq:operatorEquation}]\label{thm:existenceWithL}
{Under the assumptions in Assumptions \ref{asss:onL}, \ref{asss:aWithoutL}, \ref{ass:basisFunctionsForRegularity}, and \ref{asss:aAndasWithL}, if $f \in L^2_H$ and $u_0 \in V_0$, the unique solution $u$ of \eqref{eq:operatorEquation} from Theorem \ref{thm:existenceWithoutL} satisfies the regularity $u \in W(V,H)$ and the estimate}
\begin{equation*}
\norm{u}{W(V,H)} \leq C\left(\norm{u_0}{V_0}+\norm{f}{L^2_H}\right).
\end{equation*}
\end{theorem}


\section{Proof of well-posedness}\label{sec:proofs}
We use a generalisation of the Lax--Milgram theorem sometimes called the Banach--Ne\v{c}as--Babu\v{s}ka theorem  \cite[\S 2.1.3]{guermond} to establish existence.
\begin{theorem}[Banach--Ne\v{c}as--Babu\v{s}ka]\label{thm:bnb}
Let $X$ be a Banach space and let $Y$ be a reflexive Banach space. Suppose $d(\cdot,\cdot)\colon X \times Y \to \mathbb{R}$ is a bounded bilinear form and $f \in Y^*$. Then there is a unique solution $x \in X$ to the problem
\[d(x,y) = \langle f, y \rangle_{Y^*,Y}\qquad\text{for all $y \in Y$}\]
satisfying
\begin{equation}\label{eq:wellPosednessBNB}
\norm{x}{X} \leq C\norm{f}{Y^*}
\end{equation}
if and only if
\begin{enumerate}
\item There exists $\alpha > 0$ such that
\[\inf_{x \in X}\sup_{y \in Y}\frac{d(x,y)}{\norm{x}{X}\norm{y}{Y}} \geq \alpha\tag{``inf-sup condition"}.\]
\item For arbitrary $y \in Y$, if \[ d(x,y) = 0\text{ holds for all $x \in X$},\]
then $y=0.$
\end{enumerate}
Moreover, the estimate \eqref{eq:wellPosednessBNB} holds with the constant $C=\frac{1}{\alpha}.$
\end{theorem}
Recall the equation \eqref{eq:operatorEquation}:
\begin{equation*}
\begin{aligned}
L\dot u + A u + \symbolForBigC u &= f\qquad\text{in $L^2_{V^*}$}\\
u(0) &= u_0
\end{aligned}
\end{equation*}
where $f \in L^2_{V^*}$ and $u_0 \in H_0.$  By considering a suitable initial value problem on a fixed domain
we know that there is a function $y \in \mathcal W(V_0,V_0^*)$ with $y(0) = u_0$ and
\[\norm{y}{\mathcal W(V_0,V_0^*)} \leq C\norm{u_0}{H_0}.\]
Then the function
$\tilde y(\cdot)  = \phi_{(\cdot)}y(\cdot)$
is such that $\tilde y \in W(V,V^{*})$ with $\tilde y(0) = u_0$. So then we can transform \eqref{eq:operatorEquation} into a PDE with zero initial condition if we set $w = u - \tilde y$:
\begin{equation}\label{eq:necasPDEzero}
\begin{aligned}
L\dot {w} + A w + \symbolForBigC w &= \tilde f\\
w(0) &= 0
\end{aligned}\tag{$\textbf{P}_{\textbf{0}}$}
\end{equation}
where $\tilde f :=  f-L\md \tilde y - A\tilde y - \symbolForBigC\tilde y \in L^2_{V^*}$. It is clear that well-posedness of \eqref{eq:necasPDEzero} translates into well-posedness of \eqref{eq:operatorEquation}. The idea is to apply Theorem \ref{thm:bnb} to the problem \eqref{eq:necasPDEzero} with $X=W_0(V,V^*)$, $Y=L^2_V$, and the bilinear form
\[d(u,v) = \langle L\dot u, v \rangle_{L^2_{V^*}, L^2_V} + \langle Au, v \rangle _{L^2_{V^*}, L^2_V}+ \langle \symbolForBigC u, v \rangle_{L^2_{V^*}, L^2_V}.\]
\begin{remark}
The space $W_0(V,V^*)$ is indeed a Hilbert space because by Lemma \ref{lem:initialConditionsWellDefined}, it is a closed linear subspace of $W(V,V^*)$.
\end{remark}
The arguments in the next two lemmas follow \S 4 in \cite{reusken}. See also \cite[\S 6.1.2]{guermond}.

\begin{lemma}\label{lem:necaseInfSup}For all $w \in W_0(V,V^*)$, there exists a function $v_w \in L^2_V$ such that
\begin{align*}
\langle L\dot w, v_w\rangle_{L^2_{V^*}, L^2_V} +\langle Aw,v_w\rangle_{L^2_{V^*}, L^2_V}+\langle \symbolForBigC w&,v_w\rangle_{L^2_{V^*}, L^2_V}\geq C\norm{w}{W(V,V^*)}\norm{v_w}{L^2_V}.
\end{align*}
\end{lemma}
\begin{proof}This proof requires two estimates.
\paragraph{First estimate}Let $w \in W_0(V,V^*)$ and set $w_\gamma(t) = e^{-\gamma t}w(t).$ Note that $w_{\gamma} \in W_0(V,V^*)$ too with
$\dot w_{\gamma}(t) = e^{-\gamma t}\dot w(t) - \gamma w_{\gamma}(t),$
so
\[\langle L(t)\dot w_\gamma(t), w(t)\rangle_{\Vmt, \Vt} = \langle L(t)\dot w(t)- \gamma L(t)w(t), w_{\gamma}(t)\rangle_{\Vmt,\Vt}.\]
Rearranging, integrating, and then using \eqref{eq:assDifferentiabilityOfl}:
\begin{align}
 \langle L\dot w, w_\gamma \rangle_{L^2_{V^*}, L^2_V} &=\frac{1}{2}\left(\langle L\dot w, w_\gamma \rangle_{L^2_{V^*}, L^2_V}  + \langle L\dot w_\gamma, w\rangle_{L^2_{V^*}, L^2_V}\right) + \frac{1}{2}\gamma (Lw, w_\gamma)_{L^2_H} \label{eq:halfIdentity}\\
\nonumber &=\frac{1}{2}\int_0^T\frac{d}{dt}(L(t)w(t),w_\gamma(t))_{\Ht} - \frac{1}{2}\langle Mw,w_\gamma\rangle_{L^2_{V^*}, L^2_V}\\
\nonumber &\quad + \frac{1}{2}\gamma (Lw, w_\gamma)_{L^2_H}\\
\nonumber &\geq - \frac{1}{2}\langle Mw,w_\gamma\rangle_{L^2_{V^*}, L^2_V} + \frac{1}{2}\gamma (Lw, w_\gamma)_{L^2_H}
\end{align}
as $(L(T)w(T),w_{\gamma}(T))_{H(T)} \geq 0$ by \eqref{eq:assCoercivityOfl}. Hence
\begin{align}
\nonumber &\langle L\dot w, w_\gamma \rangle_{L^2_{V^*}, L^2_V} + \langle Aw, w_\gamma \rangle_{L^2_{V^*}, L^2_V} + \langle \symbolForBigC w, w_\gamma \rangle_{L^2_{V^*}, L^2_V}
\\
\nonumber &\geq \langle Aw, w_\gamma \rangle_{L^2_{V^*}, L^2_V} + \langle \symbolForBigC w,w_{\gamma} \rangle_{L^2_{V^*}, L^2_V} -\frac{1}{2} \langle Mw,w_{\gamma}\rangle_{L^2_{V^*}, L^2_V}+ \frac{1}{2}\gamma (Lw, w_\gamma)_{L^2_H}\\
\nonumber &\geq \int_0^T e^{-\gamma t}\left(C_1\norm{w(t)}{\Vt}^2 - C_2\norm{w(t)}{\Ht}^2\right)-\frac{1}{2}\int_0^T C_3e^{{-\gamma t}}\norm{w(t)}{H(t)}^2\\
&\quad + \frac{\gamma C_4 }{2}\int_0^T e^{-\gamma t}\norm{w(t)}{H(t)}^2\tag{by the coercivity of $A(t)$ and $L(t)$ and the boundedness of $\symbolForBigC (t)$ and $M(t)$}\\
\nonumber &= C_1\int_0^T e^{-\gamma t}\norm{w(t)}{\Vt}^2 + \frac{\gamma C_4 -C_3 - 2C_2}{2}\int_0^T e^{-\gamma t}\norm{w(t)}{H(t)}^2\\
&\geq e^{-\gamma T}C_1\norm{w}{L^2_V}^2\label{eq:result1}\tag{E1}
\end{align}
with the final inequality holding if we choose $\gamma$ such that $\gamma C_4 > C_3 + 2C_2$. Note that we used Young's inequality in conjunction with the boundedness of $M(t)$ above.
\paragraph{Second estimate}Now, by the Riesz representation theorem, there exists $z \in L^2_V$ such that
\begin{equation}\label{eq:rrt1}
\langle L\dot w, v \rangle_{L^2_{V^*}, L^2_V} = (z,v)_{L^2_V}\qquad\text{for all $v \in L^2_V$}
\end{equation}
with $\norm{z}{L^2_V} = \norm{L\dot w}{L^2_{V^*}}$. We have
\begin{align}
\nonumber \langle L\dot w + Aw + \symbolForBigC w, z \rangle_{L^2_{V^*}, L^2_V} &\geq \norm{z}{L^2_V}^2 - C_5\int_0^T \norm{w(t)}{\Vt}\norm{z(t)}{\Vt}\tag{by \eqref{eq:rrt1} and the bounds on $A$ and $\symbolForBigC$}\\
\nonumber &\geq C_6\norm{z}{L^2_V}^2 - C_7\norm{w}{L^2_V}^2\tag{using Young's inequality}\\
&= C_6\lVert{L\dot w}\rVert_{L^2_{V^*}}^2 - C_7\norm{w}{L^2_V}^2\label{eq:result2}\tag{E2}.
\end{align}
\paragraph{Combining the estimates}Estimate \eqref{eq:result2} gives us control of $L\dot w$ at the expense of $w$, but the latter is controlled by estimate \eqref{eq:result1}. So let us put $v_w := z+\mu w_\gamma$ where $\mu > 0$ is a constant to be determined and consider:
\begin{align*}
\langle L\dot w, v_w \rangle_{L^2_{V^*}, L^2_V} &+ \langle Aw, v_w\rangle_{L^2_{V^*}, L^2_V} + \langle \symbolForBigC w, v_w \rangle_{L^2_{V^*}, L^2_V}\\
&\geq C_6\norm{L\dot w}{L^2_{V^*}}^2 - C_7\norm{w}{L^2_V}^2 + \mu e^{-\gamma T}C_1\norm{w}{L^2_V}^2\\
&\geq C_6\norm{L\dot w}{L^2_{V^*}}^2 + C_8\norm{w}{L^2_V}^2\tag{if $\mu$ is large enough}\\
&\geq C_9\norm{w}{W(V,V^*)}^2
\end{align*}
thanks to \eqref{eq:assLvstarInL2Vstar}. Finally, because
\begin{align*}
\norm{v_w}{L^2_V} &\leq \norm{z}{L^2_V} + \mu\norm{w_\gamma}{L^2_V}\\
&= \norm{L\dot w}{L^2_{V^*}} + \mu\left(\int_0^T |e^{-\gamma t}|^2\norm{w(t)}{V(t)}^2\right)^{\frac 1 2}\\
&\leq \norm{L\dot w}{L^2_{V^*}} + \mu\norm{w}{L^2_V}\\
&\leq C_{10}\norm{w}{W(V,V^*)}\tag{by \eqref{eq:assLvstarInL2Vstar}}
\end{align*}
we end up with
\begin{align*}
\langle L\dot w, v_w \rangle_{L^2_{V^*}, L^2_V} + \langle Aw, v_w\rangle_{L^2_{V^*}, L^2_V}+ \langle \symbolForBigC w, v_w\rangle_{L^2_{V^*}, L^2_V}\geq C\norm{w}{W(V,V^*)}\norm{v_w}{L^2_V}.
\end{align*}
\end{proof}

\begin{lemma}\label{lem:necasCondition}
If given arbitrary ${v} \in L^2_V$, the equality
\begin{equation}\label{eq:rts}
\langle L\dot w, v \rangle_{L^2_{V^*}, L^2_V} + \langle Aw, v\rangle_{L^2_{V^*}, L^2_V} + \langle \symbolForBigC w, v \rangle_{L^2_{V^*}, L^2_V} = 0
\end{equation}
holds for all $w \in W_0(V,V^*)$, then necessarily ${v}=0.$
\end{lemma}
\begin{proof}
Define the operator $\tilde{A}(t)\colon \Vt \to \Vmt$ by 
\[\langle \tilde{A}(t)v(t), \eta (t) \rangle_{V^*(t),V(t)} := \langle A(t) \eta(t), v(t) \rangle_{V^*(t),V(t)}\]
and identify $(\tilde{A}v)(t) = \tilde{A}(t)v(t).$ Take $w=\eta \in \mathcal{D}_V$ in \eqref{eq:rts} and rearrange to give
\begin{align*}
(L\dot \eta,v)_{L^2_H}=(Lv, \dot \eta)_{L^2_H} &= -\langle \tilde{A}{v}, \eta \rangle_{L^2_{V^*}, L^2_V} - \langle \symbolForBigC v,\eta\rangle_{L^2_{V^*}, L^2_V}\\
&= -\langle \tilde{A}{v}-\symbolForBigC Lv + \symbolForBigC v, \eta\rangle_{L^2_{V^*}, L^2_V} - \langle \symbolForBigC Lv,\eta\rangle_{L^2_{V^*}, L^2_V}
\end{align*}
where we used the symmetric property of $L(t)$. (We could not simply have used $A$ in place of $\tilde{A}$ above because $a(t;\cdot,\cdot)$ may not be symmetric.) This tells us that $\md({Lv}) =\tilde{A}{v}-\symbolForBigC Lv+\symbolForBigC v \in L^2_{V^*}$, and so $Lv \in W(V,V^*)$ (we already have $Lv \in L^2_V$ from \eqref{eq:assLvInL2V}). So
\begin{align}
\nonumber \langle \md ( Lv), \eta \rangle_{L^2_{V^*}, L^2_V} &= \langle  (\tilde{A} -\symbolForBigC L+\symbolForBigC )v, \eta \rangle_{L^2_{V^*}, L^2_V}\qquad\text{$\forall \eta \in \mathcal D_V.$}\\
\intertext{By the density of $\mathcal{D}((0,T);V_0) \subset L^2(0,T;V_0)$, we have the density of $\mathcal D_V \subset L^2_V,$ which implies}
\langle \md (Lv), w\rangle_{L^2_{V^*}, L^2_V}&=  \langle (\tilde{A}-\symbolForBigC L+\symbolForBigC )v, w \rangle_{L^2_{V^*}, L^2_V}\qquad\text{$\forall w \in L^2_V.$}\label{eq:rts2}
\end{align}
If in particular $w \in W_0(V,V^*)$, then we can use \eqref{eq:rts} on the right hand side of \eqref{eq:rts2} to give
\begin{equation}\label{eq:weexist}
\langle L\dot w, {v} \rangle_{L^2_{V^*}, L^2_V} + \langle \md( Lv), w \rangle_{L^2_{V^*}, L^2_V} + \langle \symbolForBigC w, Lv\rangle_{L^2_{V^*}, L^2_V} = 0\quad\text{$\forall w \in W_0(V,V^*)$.}
\end{equation}
Using $(L(t)w(t), v(t))_{\Ht} = (L(t)v(t), w(t))_{\Ht}$, we have
\begin{align*}
\frac{d}{dt}(L(t)w(t), v(t))_{\Ht}&= \langle \md(L(t)v(t)), w(t) \rangle_{\Vmt, \Vt} \\
&\quad + \langle \dot w(t), L(t)v(t) \rangle_{\Vmt, \Vt}+ \langle \symbolForBigC (t)w(t), L(t)v(t)  \rangle
\end{align*}
to which an application of \eqref{eq:assSymmetricityOfL} shows us that \eqref{eq:weexist} is exactly
\[\int_0^T \frac{d}{dt}(L(t)w(t), v(t))_{\Ht} = (L(T)w(T),{v}(T))_{H(T)}=0\]
for all $w \in W_0(V,V^*)$. Thus we have shown that ${v}(T) = 0$.

Let $0>\gamma \in \mathbb R$ and  set $w(t) = v_\gamma(t) = e^{-\gamma t}{v}(t)$ in \eqref{eq:rts2} to obtain
\begin{align}\label{eq:proofExi1}
0&=\langle \md (Lv), v_\gamma \rangle_{L^2_{V^*}, L^2_V} - \langle (\tilde{A}-\symbolForBigC L+\symbolForBigC )v, v_\gamma \rangle_{L^2_{V^*}, L^2_V}.
\end{align}
We showed that $Lv \in W(V,V^*)$ earlier; by \eqref{eq:assLvInW1}, $v \in W(V,V^*)$ too, and so we can apply \eqref{eq:assProductRule} to the  first term on the right hand side of \eqref{eq:proofExi1}:
\begin{align*}
\langle \md (Lv), v_\gamma \rangle_{L^2_{V^*}, L^2_V} &=\langle \dot{L}v, v_\gamma \rangle_{L^2_{V^*}, L^2_V} + \langle L\dot v, v_\gamma \rangle_{L^2_{V^*}, L^2_V}\\
&=\langle \dot{L}v, v_\gamma \rangle_{L^2_{V^*}, L^2_V}+\frac{1}{2}\left(\langle L\dot v, v_\gamma \rangle_{L^2_{V^*}, L^2_V} + \langle L\dot v_\gamma, v\rangle_{L^2_{V^*}, L^2_V}\right)\\
&\quad+ \frac{1}{2}\gamma(Lv, v_\gamma)_{L^2_H}\tag{follows like the equation \eqref{eq:halfIdentity}}\\
&\leq \frac{1}{2}\langle \dot{L}v, v_\gamma \rangle_{L^2_{V^*}, L^2_V}-\frac{1}{2}\langle \symbolForBigC v_\gamma, Lv \rangle_{L^2_{V^*}, L^2_V}+ \frac{1}{2}\gamma(Lv, v_\gamma)_{L^2_H}\tag{since $ v(T) = 0$ and by coercivity of $L(0)$}.
\end{align*}
Note that \eqref{eq:assBoundednessOfdotL} together with Young's inequality implies 
 \begin{align*}
 \langle \dot{L}(t)v(t), v(t) \rangle_{V^*(t),V(t)} 
 & \leq C_5 \| v(t) \|_{H(t)} \| v(t) \|_{V(t)}
  \leq C_{\epsilon} \| v(t) \|_{H(t)}^2 + \epsilon \| v(t) \|^{2}_{V(t)}. 
  \end{align*}
Using this and the previous inequality, \eqref{eq:proofExi1} becomes
\begin{align*}
0&\leq \langle \dot{L}v, v_\gamma \rangle_{L^2_{V^*}, L^2_V}+\langle \symbolForBigC v_\gamma, Lv \rangle_{L^2_{V^*}, L^2_V}+ \gamma(Lv, v_\gamma)_{L^2_H}- 2\langle (\tilde{A}+\symbolForBigC )v, v_\gamma \rangle_{L^2_{V^*}, L^2_V}\\
&=\int_0^T e^{-\gamma t}\langle \dot{L}(t)v(t), v(t)\rangle_{\Vmt, \Vt}+\int_0^T e^{-\gamma t}\symbolForLittlec(t; L(t)v(t), v(t))\\
&\;\;+ \int_0^T\gamma e^{-\gamma t}(L(t)v(t), v(t))_{\Ht}- 2\int_0^T e^{-\gamma t}\langle (\tilde{A}(t)+\symbolForBigC (t))v(t), v (t) \rangle_{\Vmt, \Vt}\\
&\leq (C_1+ \gamma C_2)\int_0^T e^{-\gamma t}\norm{v(t)}{\Ht}^2 - 2C_a\int_0^T e^{-\gamma t}\norm{{v}(t)}{\Vt}^2
\end{align*}
using {the bound on $\symbolForLittlec(t;\cdot,\cdot)$} and the assumptions \eqref{eq:assBoundednessOfl}, \eqref{eq:assCoercivityOfl} and \eqref{eq:assCoercivityOfa} (coercivity). If we pick $\gamma = -\frac{C_1}{C_2}$, it follows that ${v}=0$ in $L^2_V.$
\end{proof}
\begin{proof}[Proof of Theorem \ref{thm:existenceWithoutL}]
The inf-sup condition (which is an easy consequence of Lemma \ref{lem:necaseInfSup}) in combination with Lemma \ref{lem:necasCondition} furnishes the requirements of the Banach--Ne\v{c}as--Babu\v{s}ka theorem (Theorem \ref{thm:bnb}) thus yielding the existence and uniqueness of a solution $w \in W_0(V,V^*)$ to
\begin{equation*}
\begin{aligned}
L\dot {w} + A w + \symbolForBigC w &= \tilde f\\
w(0) &= 0
\end{aligned}
\end{equation*}
where $\tilde f \in L^2_{V^*}$ is arbitrary. Hence, we have well-posedness of \eqref{eq:necasPDEzero} with the estimate
\[\norm{w}{W(V,V^*)} \leq C\lVert {\tilde f}\rVert_{L^2_{V^*}}.\]
From this well-posedness result, we also obtain unique solvability of \eqref{eq:operatorEquation} by setting $u = w+\tilde{y}$ (note that $w$ depends on $\tilde{y}$), with the solution $u \in W(V,V^*)$ satisfying 
\[\norm{u}{W(V,V^*)} \leq C\left(\norm{f}{L^2_{V^*}} + \norm{u_0}{H_0}\right).\]
\end{proof}

\section{Galerkin approximation}\label{sec:galerkinApproximations}
In this section we abstract the pushed-forward Galerkin method used in \cite{dziuk_elliott} for the advection-diffusion equation on an evolving hypersurface.
\subsection{Finite-dimensional spaces}
{Let $\{\chi^0_j\}_{j \in \mathbb{N}}$ be the basis of $V_0$ described in Assumption \ref{ass:basisFunctionsForRegularity}.} We can turn this into a basis of $V(t)$ with the help of the continuous map $\phi_t.$ 
\begin{lemma}
With $\chi_j^t := \phi_t(\chi_j^0)$ for each $j \in \mathbb{N}$, the set $\{\chi_j^t\}_{j \in \mathbb{N}}$ is a countable basis of $\Vt$.
\end{lemma}
The next result is an extremely useful property of the basis functions following from Remark \ref{zeromatderiv} (see \cite{dziuk_elliott} for the finite element analogue).
\begin{lemma}[Transport property of basis functions]
The basis $\{\chi_j^t\}_{j \in \mathbb{N}}$ satisfies the transport property
\[\dot\chi_j^t = 0.\]
\end{lemma}
We now construct the approximation spaces in which the discrete solutions lie.
\begin{definition}[Approximation spaces]
For each $N \in \mathbb{N}$ and each $t \in [0,T]$, define \[V_N(t) = \textnormal{span}\{\chi_1^t, ..., \chi_N^t\} \subset \Vt.\] 
Clearly $V_N(t) \subset V_{N+1}(t)$ and $\bigcup_{j \in \mathbb{N}}V_j(t)$ is dense in $V(t)$. 
Define
\[L^2_{V_N} = \{ u \in L^2_V \mid u(t) = \sum_{j=1}^N \alpha_j(t) \chi_j^t \textnormal{ where $\alpha_j\colon [0,T] \to \mathbb{R}$}\}.\]
Similarly, $L^2_{V_N} \subset L^2_{V_{N+1}}$, and we shall state a density result below which follows from the density of the embedding 
$\bigcup_{j \in \mathbb{N}} L^2(0,T;V_j(0)) \subset L^2(0,T;V_0)$ and from the fact that $L^2(0,T;V_j(0)) \subset L^2(0,T;V_{j+1}(0)).$
\end{definition}
\begin{lemma}The space $\bigcup_{j \in \mathbb{N}}L^2_{V_j}$ is dense in $L^2_V$.
\end{lemma}
\begin{remark}If $u \in L^2_{V_N}$ and $u(t) = \sum_{j=1}^N \alpha_j(t)\chi_j^t$
has coefficients $\alpha_j \in C^1([0,T]),$ then $u \in C^1_V$ with strong material derivative
$\dot u(t) = \sum_{j=1}^N \alpha_j'(t)\chi_j^t,$
and $\dot u \in L^2_{V_N}.$ Our Galerkin ansatz (see below) has coefficients in a slightly less convenient space.
\end{remark}
\paragraph{Galerkin ansatz.} Later on, we construct finite-dimensional solutions which have the form
\[u_N(t) = \sum_{j=1}^N u_j^N(t)\chi_j^t \in V_N(t)\]
where the $u_j^N\colon [0,T]\to\mathbb{R}$ turn out to be absolutely continuous coefficient functions with $\dot u_j^N \in L^2(0,T)$, i.e., $u_N \in \tilde C^1_V$. It holds that $u_N \in L^2_V$ 
and by definition, $u_N \in L^2_{V_N}.$ {By Remark \ref{rem:materialDerivativeOfTildeC1VFunctions}, the material derivative of $u_N$ is $\dot u_N \in L^2_{V_N}$ with $\dot u_N(t) = \sum_{j=1}^N \dot u_j^N(t)\chi_j^{t}.$}

\begin{definition}[Projection operators]\label{defn:projectionOperator}
For each $t \in [0,T]$, define a projection operator $P_N^t\colon H(t) \to V_N(t)$ by the formula
\begin{equation*}
(P_N^tu-u,v_N)_{\Ht} = 0\quad\text{for all $v_N \in V_N(t)$}.
\end{equation*}
\end{definition}
It follows that $(P_N^t)^2 = P_N^t,$
\begin{align*}
\norm{P_N^tu}{H(t)} &\leq \norm{u}{H(t)}
\end{align*}
and
\begin{align}
P_N^t u \to u \quad \text{in $\Ht$}\label{eq:projectionOperatorConvergesInH}
\end{align}
for all $u \in \Ht.$ 
\begin{remark}
We could have relaxed the definition of the spaces $V_N(t)$ and instead have asked for a family of finite-dimensional 
spaces $\{V_N(0)\}_{N \in \mathbb{N}}$ such that for all $N \in \mathbb{N}$,
\begin{itemize}
\item[(i)]$V_N(0) \subset V_0$
\item[(ii)]$\text{dim}(V_N) = N$
\item[(iii)]$\bigcup_{i \in \mathbb{N}}V_i(0)$ is dense in $V_0$
\item[(iv)]For every $v \in V_0,$ there exists a sequence $\{v_N\}_{N \in \mathbb{N}}$ with $v_N \in V_N(0)$ such that $\norm{v_N - v}{V_0} \to 0$.
\end{itemize}
Furthermore, we can define the spaces $V_N(t) := \phi_t(V_N(0)).$ The continuity of the map $\phi_t$ 
implies that these spaces share the same properties (with respect to $V(t)$) as the $V_N(0)$ given above; in particular the density result 
\[\bigcup_{N \in \mathbb{N}}V_N(t) \quad \text{is dense in $V(t)$}\]
is true. Note that the basis of $V_N(t)$ does not necessarily have to be a subset of the basis of $V_{N+1}(t)$; this is the situation in finite element analysis, for example, so this relaxation can be useful 
for the purposes of numerical analysis. See \cite{dziuk_elliott}, \cite{dziuk_elliott-L2}.
\end{remark}

\subsection{Galerkin approximation of \eqref{eq:operatorEquation}}
We now proceed with the regularity result. With $f \in L^2_H$ and $u_0 \in V_0$, the finite-dimensional approximation is to find a unique $u_N \in L^2_{V_N}$ with $\dot u_N \in L^2_{V_N}$ satisfying
\begin{equation}\label{eq:galerkinODESystemWithL}
\begin{aligned}
l(t;\dot u_N(t), \chi_j^t) + a(t;u_N(t),\chi_j^t) + \symbolForLittlec(t;u_N(t), \chi_j^t) &=  (f(t), \chi_j^t)_{H(t)}\\
u_N(0) &= u_{0N}
\end{aligned}
\end{equation}
for all $j\in \{1, ..., N\}$ and for almost every $t \in [0,T]$ (cf. the equation \eqref{eq:defnWeakSolutionWithoutL}). Here, $u_{0N}$ is as in Assumption \ref{ass:basisFunctionsForRegularity}.
%
%
\begin{theorem}[Well-posedness of solutions to the finite-dimensional problem]\label{thm:eauFD}
Under the hypotheses of Theorem \ref{thm:existenceWithL}, there exists a unique $u_N \in L^2_{V_N}$ with $\dot u_N \in L^2_{V_N}$ satisfying the problem \eqref{eq:galerkinODESystemWithL}. With $u_N(t) = \sum_{i=1}^N u_i^N(t)\chi_i^t$, the coefficient functions satisfy
\begin{align*}
&u_i^N \in AC([0,T])\\
&\dot u_i^N \in L^2(0,T).
\end{align*}
for all $i \in \{1, ..., N\}.$
\end{theorem}

\begin{proof}
Substitute $u_N(t) = \sum_{i=1}^N u_i^N(t)\chi_i^t$ into \eqref{eq:galerkinODESystemWithL} to yield
\begin{align}
&\sum_{i=1}^N\dot u_i^N(t)l_{ij}(t) + u_i^N(t)(a_{ij}(t) + c_{ij}(t)) = f_{j}(t)\label{eq:proofGalerkinODESystem4}
\end{align}
with $l_{ij}(t) = l(t;\chi_i^t, \chi_j^t)$, $a_{ij}(t) = a(t; \chi_i^t,\chi_j^t)$, $\symbolForLittlec_{ij}(t) = \symbolForLittlec(t; \chi_i^t,\chi_j^t)$ 
and $f_{j}(t) = (f(t), \chi_j^t)_{H(t)}.$ Defining the vectors $(\mathbf{u^N(t)})_i = u_i^N(t)$ and $(\mathbf{F(t)})_i=f_{i}(t),$ and
matrices $(\mathbf{L(t)})_{ij} = l_{ji}(t),$ $(\mathbf{A(t)})_{ij} = a_{ji}(t)$, and $(\mathbf{\symbolForBigC(t)})_{ij} = \symbolForLittlec_{ji}(t)$, we can write \eqref{eq:proofGalerkinODESystem4} in matrix-vector form as
\[\mathbf{L(t)}\mathbf{\dot u^N(t)} + (\mathbf{A(t)+\symbolForBigC(t))u^N(t)} = \mathbf{F(t)}.\]
Elementary considerations show that $\mathbf{L(\cdot)}^{-1} \in L^\infty(0,T;\mathbb{R}^{N\times N}),$ so we can rearrange the system to
\begin{equation}\label{eq:proofGalerkinODESystem2}
\mathbf{\dot u^N(t)} + \mathbf{L(t)}^{-1}\mathbf{(A(t)+\symbolForBigC(t))u^N(t)} = \mathbf{L(t)}^{-1}\mathbf{F(t)}.
\end{equation}
Note that $\mathbf{F(\cdot)} \in L^2(0,T;\mathbb{R}^N)$
and $\mathbf{A(\cdot)+\symbolForBigC(\cdot)} \in L^\infty(0,T;\mathbb{R}^{N\times N})$. So the coefficients 
of \eqref{eq:proofGalerkinODESystem2} are all measurable in time, and we can apply standard theory that 
guarantees the existence and uniqueness of $u_j^N \in  AC([0,T])$ with $\dot u_j^N \in L^2(0,T)$, and thus the existence and uniqueness of $u_N$. The function $u_N \in \tilde{C}^1_V$ is a 
solution in the sense that the derivative $\dot u_N$ exists almost everywhere and the ODE is satisfied almost everywhere.
\end{proof}

The Galerkin approximation is equivalent to the discrete equation
\begin{equation}\label{eq:galerkinEquationRegularity}
l(t;\dot u_N(t), v_N(t)) + a(t;u_N(t),v_N(t)) + \symbolForLittlec(t;u_N(t),v_N(t)) = (f(t), v_N(t))_{H(t)} \tag{$\textbf{P}_{\textbf{d}}'$}
\end{equation}
for all $v_N \in L^2_{V_N}$. We look for \emph{a priori} estimates on $u_N$ and $\dot u_N$ in appropriate norms.

\begin{lemma}[{A priori} estimate on $u_N$]\label{lem:aPrioriOnUN}Under the hypotheses of Theorem \ref{thm:existenceWithL}, the following estimate holds:
\begin{equation*}
\norm{u_N}{L^2_V} \leq C\left(\norm{u_0}{H_0}+\norm{f}{L^2_{V^{*}}}\right).
\end{equation*}
\end{lemma}
\begin{remark}{This \emph{a priori} estimate is still valid under the hypotheses of Theorem \ref{thm:existenceWithoutL} if we pick $u_N(0)$ differently. See \S \ref{sec:proofExistenceSecondGalerkin} for more.}
\end{remark}
For convenience, we shall sometimes omit the argument $(t)$ in expressions like $u_N(t)$. It should be clear 
from the context the instances in which we are referring to an element of $H(t)$ as opposed to an element of $L^2_H.$
\begin{proof}[Proof of Lemma \ref{lem:aPrioriOnUN}]
Picking $v_N = u_N$ in \eqref{eq:galerkinEquationRegularity} gives
\begin{equation*}
l(t;\dot u_N, u_N)+ a(t;u_N,u_N) +\symbolForLittlec(t;u_N,u_N) = (f,u_N)_{H(t)},
\end{equation*}
which we integrate in time and apply the transport identity \eqref{eq:assDifferentiabilityOfl} to yield
\begin{align*}
\int_0^T\frac{1}{2}\frac{d}{dt}l(t;u_N,u_N) +a(t;u_N,u_N)+\symbolForLittlec(t;u_N,u_N)&-\frac{1}{2}m(t;u_N, u_N)\\
&= \int_0^T(f,u_N)_{H(t)}.
\end{align*}
Using the boundedness \eqref{eq:assBoundednessOfl} and coercivity \eqref{eq:assCoercivityOfl} of $l(t;\cdot,\cdot)$ leads to
\begin{align*}
\frac{C_c}{2}\norm{u_N(T)}{H(T)}^2 + \int_0^Ta(t;u_N,u_N) &+\int_0^T\symbolForLittlec(t;u_N,u_N)-\frac{1}{2}\int_0^Tm(t;u_N,u_N)\\ &\leq \int_0^T \langle f,u_N\rangle_{\Vmt, \Vt} + \frac{C_b}{2}\norm{u_N(0)}{H_0}^2,
\end{align*}
to which we use \eqref{eq:assCoercivityOfa} (the coercivity of $a(t;\cdot,\cdot)$), the boundedness of 
$\symbolForLittlec(t;\cdot,\cdot)$ and $m(t;\cdot,\cdot)$, and Young's inequality with $\epsilon > 0$:
\begin{align*}
 \frac{C_c}{2}\norm{u_N(T)}{H(T)}^2 + \frac{C_1}{2}\norm{u_N}{L^2_V}^2
&\leq \frac{C_2}{2}\norm{u_N}{L^2_H}^2+\frac{1}{2\epsilon}\norm{f}{L^2_{V^*}}^2 + \frac{\epsilon}{2}\norm{u_N}{L^2_V}^2\\
&\quad+ \frac{C_b}{2}\norm{u_N(0)}{H_0}^2.
\end{align*}
That is,
\begin{equation}\label{eq:aPrioriEstimateUN1}
C_c\norm{u_N(T)}{H(T)}^2 + (C_1-\epsilon)  \norm{u_N}{L^2_V}^2 \leq \frac{1}{\epsilon}\norm{f}{L^2_{V^*}}^2 + C_2 \norm{u_N}{L^2_H}^2 + C_b\norm{u_N(0)}{H_0}^2,
\end{equation}
and if $\epsilon$ is picked small enough, we can discard the second term on the left hand side and then an application of Gronwall's inequality  yields
\[\norm{u_N(t)}{H(t)}^2 \leq C_4\left(\norm{f}{L^2_{V^*}}^2+\norm{u_N(0)}{H_0}^2\right).\]
Using this on \eqref{eq:aPrioriEstimateUN1} and 
{utilising \eqref{eq:assu0NBoundedInH0} produces the desired estimate.}
\end{proof}
\begin{lemma}[{A priori} estimate on $\dot u_N$]\label{lem:aPrioriOnDotUN} 
Under the hypotheses of Theorem \ref{thm:existenceWithL}, the following estimate holds:
\begin{equation*}
\norm{\dot u_N}{L^2_H}
\leq C\left(\norm{u_0}{V_0}+\norm{f}{L^2_H}\right).
\end{equation*}
\end{lemma}
\begin{proof}
In \eqref{eq:galerkinEquationRegularity}, pick $v_N = \dot u_N$ and use \eqref{eq:assCoercivityOfl} to get
\begin{equation}\label{eq:proofAPrioriDotUN}
C_1\norm{\dot u_N}{\Ht}^2 + a_s(t;u_N,\dot u_N)+ a_n(t;u_N,\dot u_N) + \symbolForLittlec(t;u_N, \dot u_N) \leq (f,\dot u_N)_{\Ht}.
\end{equation}
Then using assumption \eqref{eq:assDifferentiabilityOfas}, \eqref{eq:proofAPrioriDotUN} is
\begin{align*}
C_1\norm{\dot u_N}{\Ht}^2 + \frac{1}{2}\frac{d}{dt}a_s(t;u_N,u_N) &\leq (f, \dot u_N)_{\Ht} + \frac{1}{2}r(t;u_N) -a_n(t; u_N, \dot u_N)\\
&\quad - \symbolForLittlec(t;u_N, \dot u_N).
\end{align*}
Integrating this yields
\begin{align*}
C_1&\int_0^T\norm{\dot u_N}{\Ht}^2 + \frac{1}{2}a_s(T;u_N(T),u_N(T))\\
&\leq \int_0^T (f,\dot u_N)_{\Ht}+ \frac{1}{2}\int_0^Tr(t;u_N)-\int_0^T a_n(t;u_N,\dot u_N) -\int_0^T\symbolForLittlec(t;u_N, \dot u_N)\\
&\quad+ \frac{1}{2}a_s(0;u_N(0),u_N(0)).
\end{align*}
where we used \eqref{eq:assAbsCtyOfas}. With \eqref{eq:assPositivityOfas} (positivity of $a_s(t;\cdot,\cdot)$), the bound \eqref{eq:assBoundednessOfas} 
on $a_s(0;\cdot,\cdot)$, the bound \eqref{eq:assBoundednessOfF} on $r(t;\cdot)$, the 
bound \eqref{eq:assBoundednessOfan} on $a_n(t;\cdot,\cdot),$ the bound on $\symbolForLittlec(t;\cdot, \cdot)$ and Young's inequality with $\epsilon > 0$ and $\delta > 0$, we get
\begin{align*}
C_1\norm{\dot u_N}{L^2_H}^2
&\leq \frac{1}{2\delta} \norm{f}{L^2_H}^2 + \left(C_2  + \frac{C_3}{2\epsilon}\right)\norm{u_N}{L^2_V}^2 + \frac{(\delta + C_3\epsilon)}{2}\norm{\dot u_N}{L^2_H}^2\\
&\quad+ C_4\norm{u_N(0)}{V_0}^2\\
&\leq \frac{1}{2\delta} \norm{f}{L^2_H}^2 + C_5\left(C_2+\frac{C_3}{2\epsilon}\right)(\norm{u_N(0)}{H_0}^2+\norm{f}{L^2_H}^2)\\
&\quad  + \frac{(\delta + C_3\epsilon)}{2}\norm{\dot u_N}{L^2_H}^2+ C_4\norm{u_N(0)}{V_0}^2 \tag{by the first \emph{a priori} bound}\\
&=\left(\frac{1}{2\delta} +C_5\left(C_2 +\frac{C_3}{2\epsilon}\right)\right)\norm{f}{L^2_H}^2 + C_5\left(C_2+\frac{C_3}{2\epsilon}\right)\norm{u_N(0)}{H_0}^2\\
&\quad  + \frac{(\delta + C_3\epsilon)}{2}\norm{\dot u_N}{L^2_H}^2 + C_4\norm{u_N(0)}{V_0}^2.
\end{align*}
If $\epsilon$ and $\delta$ are small, we can obtain the estimate by using the assumption \eqref{eq:assInitialDataApproximationBoundedInV}.
\end{proof}

\subsection{Proof of regularity}
By the estimates above, we obtain the convergence
\begin{equation}\label{eq:convergencesWithL}
\begin{aligned}
u_N &\weaklyto u \quad \text{in $L^2_V$}\\
\dot u_N &\weaklyto w \quad \text{in $L^2_{H}$}
\end{aligned}
\end{equation}
for some $u \in L^2_V$ and $w \in L^2_{H}$ and for a subsequence which we have relabelled. Now we show that in fact, $w = \dot u$.
\begin{lemma}In the context of the above convergence results, $w=\dot{u}$.
\end{lemma}
\begin{proof}
By definition
\begin{equation}\label{eq:proofConvergenceOfMaterialDerivative}
\int_0^T \langle \dot{u}_N(t), \eta(t)\rangle_{\Vmt, \Vt} =-\int_0^T (u_N(t), \dot{\eta}(t))_{\Ht} - \int_0^T \symbolForLittlec(t;u_N(t), \eta(t))
\end{equation}
holds for all $\eta \in \mathcal{D}_V(0,T).$ Since $\langle \cdot, \eta \rangle_{L^2_{V^*}, L^2_V}$, $(\cdot, \dot \eta)_{L^2_H}$, and $\langle \symbolForBigC(\cdot), \eta \rangle_{L^2_{V^*}, L^2_V}$
are all elements of $L^2_{V^*}$, using \eqref{eq:convergencesWithL}, we can pass to the limit in \eqref{eq:proofConvergenceOfMaterialDerivative} to obtain
\[\int_0^T \langle w(t), \eta(t)\rangle_{\Vmt, \Vt} =-\int_0^T (u(t), \dot{\eta}(t))_{\Ht} - \int_0^T \symbolForLittlec(t;u(t), \eta(t)),\]
i.e., $w = \dot{u}.$
\end{proof}
\begin{proof}[Proof of Theorem \ref{thm:existenceWithL}]
Given $v \in L^2_V,$ by density, there is a sequence $\{v_M\}$ with $v_M \in L^2_{V_M}$ for each $M$ such that 
\begin{align*}
v_M(t) = \sum_{j=1}^M \alpha^M_j(t)\chi_j^t \qquad \text{and} \qquad \norm{v_M - v }{L^2_V} \to 0.
\end{align*}
For $j=1, ..., N,$ consider the equation \eqref{eq:galerkinODESystemWithL}:
\[l(t;\dot u_N(t), \chi_j^t) + a(t;u_N(t),\chi_j^t) + \symbolForLittlec(t;u_N(t), \chi_j^t) = (f(t), \chi_j^t)_{\Ht}.\]
If $M \leq N$, then $v_M \in L^2_{V_N}$ and we multiply the above by $\alpha^M_j(t)$ and sum up to get
\begin{align*}
l(t;\dot u_N(t), v_M(t)) + a(t;u_N(t),v_M(t)) + \symbolForLittlec&(t;u_N(t), v_M(t))= (f(t), v_M(t))_{\Ht}.
\end{align*}
By the bounds on the respective bilinear forms, we see that
$\langle L(\cdot), v_M \rangle_{L^2_{V^*}, L^2_V}$, 
$\langle A(\cdot), v_M \rangle_{L^2_{V^*}, L^2_V}$, and 
$\langle \symbolForBigC(\cdot), v_M \rangle_{L^2_{V^*}, L^2_V}$
are elements of $L^2_{V^*}$, 
so we obtain after integrating the above equation and taking the limit as $N \to \infty$ the equation
\begin{align*}
\int_0^T l(t;\dot u(t), v_M(t)) + a(t;u(t),v_M(t)) + &\symbolForLittlec(t;u(t), v_M(t))\\
&= \int_0^T (f(t), v_M(t))_{\Ht}.
\end{align*}
Now note that as a function of $v_M$, each term in the above equation is an element of $L^2_{V^*}$ again because of the bounds on $l(t;\cdot,\cdot)$, $a(t;\cdot,\cdot)$ and $\symbolForLittlec(t;\cdot,\cdot).$
So we send $M \to \infty$, bearing in mind that $v_M$ strongly converges to $v$ in $L^2_V$:
\begin{align*}
\int_0^T l(t;\dot u(t), v(t)) + a(t;u(t),v(t)) + \symbolForLittlec&(t;u(t), v(t)) = \int_0^T ( f(t), v(t) )_{\Ht}.
\end{align*}
Hence $u \in W(V,H)$ is a solution. Let us now check the initial condition. Let $w \in V_0$, take $\zeta \in C^1[0,T]$ with $\zeta(T) = 0$, and set $v(t) = \zeta(t)\phi_tw$; we see that $v \in L^2_V.$ Since $w \in \Vs$, there exist coefficients $\alpha_j$ with $w = \sum_{j=1}^\infty \alpha_j \chi_j^0$, so
\begin{equation}\label{eq:proofOfIC1x}
v(t) = \zeta(t)\sum_{j=1}^\infty \alpha_j \chi_j^t.
\end{equation}
The sequence $\{v_N\}_{N \in \mathbb{N}}$ defined by
\begin{equation}\label{eq:proofOfIC2x}
v_N(t) =\zeta(t)\sum_{j=1}^N\alpha_j \chi_j^t
\end{equation}
is such that $v_N \in L^2_{V_N}$ and satisfies
$\norm{v_N-v}{L^2_V} \to 0$ 
by definition of $w$ as an infinite sum.
Similarly, we can show that $\dot v_N \to \dot v$ in $L^2_V.$ Using the identity \eqref{eq:assDifferentiabilityOfl} with $v$ chosen as in \eqref{eq:proofOfIC1x}, we see that
\begin{align}
\nonumber -l(0;u(0),  v(0))+&\int_0^Ta(t;u(t),v(t)) + \symbolForLittlec(t;u(t),v(t))\\
&= \int_0^T(f(t), v(t))_{H(t)} +l(t;u(t),\dot v(t))+ m(t;u(t),v(t)).\label{eq:iccompx}
\end{align}
Similarly, with $v_N$ chosen as in \eqref{eq:proofOfIC2x} in the Galerkin equation \eqref{eq:galerkinEquationRegularity}, to which we again apply \eqref{eq:assDifferentiabilityOfl} and integrate to obtain
\begin{align*}
-l(0;u_N(0),  &v_N(0)) + \int_0^Ta(t;u_N(t),v_N(t)) + \symbolForLittlec(t;u_N(t),v_N(t))\\
&\;\;= \int_0^T(f(t), v_N(t))_{H(t)}+l(t;u_N(t),\dot v_N(t))+m(t;u_N(t),v_N(t)).
\end{align*}
Using $u_N \weaklyto u,$ $v_N \to v,$ $\dot v_N \to \dot v,$ {and \eqref{eq:assu0NConvergesInV0}}, we may pass to the limit in this equation
and a comparison of the result to \eqref{eq:iccompx} will tell us that
\[l(0;u_0-u(0), \zeta(0)w )=0.\]
The arbitrariness of $w \in V_0$ and the density of $V_0$ in $H_0$ 
yield the result. 

The stability estimate follows directly from the estimates in Lemmas \ref{lem:aPrioriOnUN} and \ref{lem:aPrioriOnDotUN}. That the solution is unique follows by a straightforward adaptation of the standard technique.

\end{proof}

\subsection{Second sketch proof of existence}\label{sec:proofExistenceSecondGalerkin}
\begin{proof}[Sketch proof of Theorem \ref{thm:existenceWithoutL}]
{We can take the Galerkin approximation of \eqref{eq:defnWeakSolutionSecond} and instead of picking the initial data of $u_N$ to be $u_{0N}$ we pick $u_N(0) = P_N^0(u_0)$, where $P_N^0$ is the projection operator in Definition \ref{defn:projectionOperator}.} We still obtain the uniform bound of Lemma \ref{lem:aPrioriOnUN}, which
implies that
\begin{equation}
\begin{aligned}\label{eq:convergencesWithoutL}
u_N &\weaklyto u \quad \text{in $L^2_V$}
\end{aligned}
\end{equation}
for some $u \in L^2_V.$ An equation similar to \eqref{eq:galerkinEquationRegularity} will hold, in which we pick $v_N(t) = \chi_j^t$, where $j \in \{0, ..., N\}$, and multiplying by $\zeta \in C^1[0,T]$ with $\zeta(T) = 0$, we get
\begin{align*}
l(t;\dot u_N, \zeta\chi_j) + a(t;u_N,\zeta\chi_j) + \symbolForLittlec&(t;u_N,\zeta\chi_j)= \langle f, \zeta\chi_j \rangle_{\Vmt, \Vt},
\end{align*}
and then integrating, using the transport formula \eqref{eq:assDifferentiabilityOfl}, and passing to the limit 
with the help of \eqref{eq:convergencesWithoutL} and \eqref{eq:projectionOperatorConvergesInH}:
\begin{align}
\nonumber &\int_0^T l(t;u(t), \zeta'(t)\chi_j^t) + a(t;u(t),\zeta(t)\chi_j^t) +\symbolForLittlec(t;u(t),\zeta(t)\chi_j^t)-m(t;u(t),\zeta(t)\chi_j^t)\\
&\quad = -\int_0^T\langle f(t), \zeta(t)\chi_j^t \rangle_{\Vmt, \Vt}-l(0;u_0,\zeta(0)\chi_j^0).\label{eq:proofExistenceWithoutL1}
\end{align}
Now, 
we can write an arbitrary element of $V_0$ as $v = \sum_{i=1}^\infty \alpha_j \chi_j^0.$ 
By definition, the sequence
$v_n = \sum_{i=1}^n\alpha_j \chi_j^0$
converges to $v$ in $\Vs.$ 
It follows that $\phi_t v_n \to \phi_t v$ in $\Vt$. 
Letting $\zeta(0) = 0$, multiplying \eqref{eq:proofExistenceWithoutL1} by $\alpha_j$ and summing over $j$ gives us
\begin{align}
\nonumber \int_0^T &\zeta'(t) l(t;u(t),\phi_t v_n)\\
&= -\int_0^T \zeta(t) \langle f(t)-A(t)u(t)-\symbolForBigC(t)u(t)+M(t)u(t), \phi_t v_n \rangle_{\Vmt, \Vt}.\label{eq:proofExistenceWithoutL4}
\end{align}
It is not difficult to see that the dominated convergence theorem applies and we can pass to the limit in \eqref{eq:proofExistenceWithoutL4} to obtain
\begin{align*}
\int_0^T& \zeta'(t) l(t;u(t),\phi_t v)\\
&= -\int_0^T \zeta(t) \langle f(t)-A(t)u(t)-\symbolForBigC(t)u(t)+M(t)u(t), \phi_t v \rangle_{\Vmt, \Vt}.
\end{align*}
If we further let $\zeta \in \mathcal{D}(0,T)$, this is precisely the statement
\begin{align*}
\frac{d}{dt}l(t;u(t),\phi_t v) &= \langle f(t)-A(t)u(t)-\symbolForBigC(t)u(t)+M(t)u(t), \phi_t v \rangle_{\Vmt, \Vt}
\end{align*}
in the weak sense. This is true for every $v \in \Vs,$ and because $f-Au-\symbolForBigC u \in L^2_{V^*}$, by Lemma \ref{lem:temamType}, 
$L\dot u +A +\symbolForBigC u = f$
holds as an equality in $L^2_{V^*}$ with $u \in W(V,V^*)$. 
\end{proof}

%
%
\section*{Acknowledgements}Two of the authors (A.A. and C.M.E.) were participants of the Isaac Newton Institute programme \emph{Free Boundary Problems and Related Topics} (January -- July 2014) when this article was completed. A.A. was supported by the Engineering and Physical Sciences Research Council (EPSRC) Grant EP/H023364/1 within the MASDOC Centre for Doctoral Training. The authors would like to express their gratitude to the anonymous referees for their careful reading and valuable feedback.

\end{document}